\documentclass[a4paper, 11pt]{article}
\usepackage{courier}
\usepackage{blindtext}
\usepackage{amsmath}
\usepackage{amssymb}
\usepackage[colorlinks=true,breaklinks]{hyperref}
\usepackage[hyphenbreaks]{breakurl}
\usepackage{xcolor}
\definecolor{c1}{rgb}{0,0,1}
\definecolor{c2}{rgb}{0,0.3,0.9}
\definecolor{c3}{rgb}{0.3,0.9}
\hypersetup{linkcolor={c1},citecolor={c2},urlcolor={c3}}
\usepackage{enumerate}
\usepackage{todonotes}
\usepackage{makeidx}

\makeindex
\usepackage[top=3cm,
bottom=3.5cm,
left=2.8cm,right=2.8cm]{geometry}
\usepackage{fancyhdr}

\def\XXint#1#2#3{{\setbox0=\hbox{$#1{#2#3}{\int}$ }
\vcenter{\hbox{$#2#3$ }}\kern-.6\wd0}}

\usepackage{amsthm}
\theoremstyle{plain}
\newtheorem{theorem}{Theorem}[section]

\theoremstyle{definition}

\theoremstyle{lemma}

\theoremstyle{Remark}
\newtheorem{Remark}[theorem]{Remark}

\theoremstyle{proposition}
\newtheorem{proposition}[theorem]{Proposition}

\theoremstyle{corollary}
\newtheorem{corollary}[theorem]{Corollary}

\theoremstyle{example}

\theoremstyle{assumption}

\usepackage{mathrsfs}
\usepackage{systeme}
\usepackage{colonequals}

\begin{document}
\pagestyle{empty}
\title{Sharp estimates for the Fourier transform of surface-carried measures and maximal operators associated with  hypersurfaces in $\mathbb{R}^4$ with vanishing Gaussian curvature.}


\author{Isroil A.\,Ikromov\thanks{V.\,I.\,Romanovsky Institute  of Mathematics,  Uzbekistan Academy of Sciences,
University Boulevard 15, 140104, Samarkand,  Uzbekistan,
{\tt ikromov1@rambler.ru}},    \,\,\,Gayrat Toshpulatov\thanks{Institut f\"ur Analysis und Numerik, Fachbereich Mathematik und Informatik der Universit\"at M\"unster, Orl\'eans-Ring 10, 48149 M\"unster, Germany, {\tt  gayrat.toshpulatov@uni-muenster.de}}}

\maketitle

\pagestyle{plain}

\begin{abstract}
In this paper, we study problems related to harmonic analysis on hypersurfaces in $\mathbb{R}^4 $ with zero Gaussian curvature and given as  graphs of polynomial functions.  We derive  sharp uniform estimates with respect to the direction of frequencies for the Fourier transform  of measures supported on such hypersurfaces.  Additionally, we study the $L^p$-boundedness problem of maximal operators associated with hypersurfaces. We determine the exact value of  the  boundedness exponent  in terms of the heights of these hypersurfaces.

\end{abstract}
\tableofcontents
\section{Introduction}
Let $S$ be a hypersurface in $\mathbb{R}^{n+1}$ and let $\rho d\sigma$ be a surface-carried measure on $S,$ where $d\sigma$ denotes the surface measure on $S$ and $\rho\geq 0$ is a smooth density with a compact support. The Fourier transform of $\rho d\sigma$ is defined by
 \begin{equation*}
 \widehat{\rho d\sigma}(\xi)\colonequals \int_S e^{i\xi\cdot x}\rho(x) d\sigma(x),\,\,\,\,\xi\in \mathbb{R}^{n+1}.
\end{equation*}
 One of the fundamental problems  in harmonic analysis is to determine the exact asymptotic behavior of $\widehat{\rho d\sigma}(\xi)$ as  $|\xi|$ tends to infinity.  The applications range from number theory to  partial differential equations, see \cite{Stein, S.Sh}.

  Using a partition of unity argument, we can assume that $\rho$ is concentrated in a sufficiently small neighborhood of a fixed point $x^0\in S.$
It is known  that the  asymptotic behavior of $\widehat{\rho d\sigma}(\xi)$ is
invariant under translations and rotations of the ambient space. Hence, we may replace
the surface $S$ by any suitable image under a Euclidean motion of $\mathbb{R}^{n+1}.$ We may thus assume  $x^0=(0,...,0)\in \mathbb{R}^{n+1}$ and that $S$ is the graph
\begin{equation*}
S=\{(x_1,...,x_n, \phi(x_1,...,x_n))\in \mathbb{R}^{n+1}:\,x=(x_1,...,x_n)\in U\}
\end{equation*}
of a smooth function $\phi$ defined in  a sufficiently small neighborhood $U\subset \mathbb{R}^n$ of the origin such that
\begin{equation*}
\phi(0)=0,\,\,\,\,\nabla \phi(0)=0.
\end{equation*}
  This let us write  $\widehat{\rho d\sigma}(\xi)$ as an oscillatory integral
   \begin{equation}\label{Oscil}
 \widehat{\rho d\sigma}(\xi)= \int_{\mathbb{R}^n} e^{i(\xi_{n+1}\phi(x)+\xi_1x_1+...+\xi_nx_n)}\eta(x) dx,
\end{equation}
where
 $
 \eta(x)\colonequals \rho(x_1,...,x_n, \phi (x))\sqrt{1+|\nabla \phi(x)|^2} \in C^{\infty}_0(U).
 $
  If $\xi_1=...=\xi_n=0$ and $ \xi_{n+1}=:\lambda,$ then  \eqref{Oscil}   can be written in the form
\begin{equation}\label{Osc}
\int_{\mathbb{R}^n}e^{i\lambda \phi(x)}\eta (x)dx.
\end{equation}
 In one dimension $n=1$, the asymptotic behavior of the integral \eqref{Osc} as the parameter $\lambda$ tends to infinity is determined by the maximum order of vanishing of the function $\phi$ at its critical points, see \cite{Stein}.
In higher dimensions, this problem  is substantially difficult. This is connected with the multiplicity and complexity of the critical points of $\phi$ that occur when the dimension exceeds one. It is well-known that  the asymptotic behavior of the oscillatory integral \eqref{Oscil} does not change under a smooth change of variables.   Arnold conjectured in \cite{Ar} that  the asymptotic behavior of the oscillatory integral \eqref{Oscil} is completely determined by the Newton polyhedron of  $\phi$ in a so-called \emph{adapted} coordinate system (see the next section for the definition of the Newton polyhedron and others). More precisely, with $\phi$
we can associate  the so-called \emph{height} $h(\phi)$
defined in terms of the Newton polyhedron of $\phi$  when represented in  a smooth coordinate
system near the origin. An important property of this height $h(\phi)$ is that
it is invariant under local smooth changes of coordinates fixing the origin.
 Arnold expected that the leading asymptotic behavior or the main part of the asymptotic development would be determined by the height  $h(\phi)$ of $\phi$.
    Later, when $\phi$ is real-analytic, Varchenko \cite{Var} proved  that Arnold's conjecture holds true for \eqref{Osc} in two dimensions $n=2.$  In particular, Varchenko gave sufficient conditions for the adaptedness of a given coordinate system, proved the existence of an adapted coordinate system for analytic functions $\phi$ without multiple components, and showed that the leading term of the asymptotic expansion of \eqref{Osc} can determined by the height $h(\phi)$ and  the dimension of the principal face in an adapted coordinate system.     Karpushkin \cite{Kar} proved that
Varchenko's estimate  is stable under sufficiently small analytic perturbations of $\phi.$ 
Varchenko's result on the existence of an adapted coordinate system has been extended by Ikromov and M\"uller \cite{Adap}  to arbitrary real-analytic functions and arbitrary smooth functions of finite type in two dimensions, see also \cite{Greenblatt}.   Their approach is inspired by the work of Phong, Stein,  and Strum \cite{Ph.SS} on the Newton polyhedron and oscillatory integrals; see \cite{Ph.S} for applications in the context of oscillatory integral operators.
Ikromov and   M\"uller  \cite{Ik.Mu} also proved Arnold's conjecture for smooth functions of finite type in two dimensions. In this study, the asymptotic decay rates are sharp.

However, in three dimensions and higher, Arnold's conjecture 
does not hold in general. The reason is that  Varchenko \cite{Var} constructed an example in three dimensions, for which  the leading asymptotics differ from the predicted. Arnold's conjecture is currently known to hold in any dimension for smooth convex functions  of finite line type, see  \cite{Sch, I.U}.
For analytic functions, the asymptotic behavior of oscillatory integrals can be studied  based on Hironaka's theorem on the resolution of singularities \cite{Hironaka, Atiyah}, for recent results on this topic, see \cite{Pramanik}.

In three dimensions and higher, the simple description of the asymptotic behavior of integral \eqref{Oscil} remains unknown for general hypersurfaces. In this paper, we aim to shed light on the three-dimensional case $n=3$.  We investigate the asymptotic behavior of the Fourier transform $\widehat{\rho d\sigma}$ for hypersurfaces in $\mathbb{R}^4.$  
 Specifically, we focus on hypersurfaces $S\subset \mathbb{R}^4$  with zero Gaussian curvature and given as  graphs of polynomial functions. This means that we consider  hypersurfaces   of the form
 \begin{equation*}
S=\{(x_1,x_2,x_3, \phi(x_1,x_2,x_3))\in \mathbb{R}^{4}:\,x=(x_1,x_2,x_3)\in U\},
\end{equation*}
 where $U\subset \mathbb{R}^3$ is a neighborhood of the origin. We assume that $\phi$ is a polynomial function satisfying
 \begin{equation}\label{phi=0}
 \phi(0)=0,\,\,\,\nabla \phi (0)=0,
 \end{equation} and
 \begin{equation}\label{det=0}
 \det(D^2\phi(x))=0,\,\,\,\,\forall\,x\in \mathbb{R}^3.
 \end{equation}
 The condition \eqref{det=0} means that the Gaussian curvature of $S$ is identically zero. Our results concerning the oscillatory integral  \eqref{Oscil} are the following:
 \begin{itemize}
 \item In Theorem \ref{th:adap}, we show that there is an adapted coordinate system for any polynomial function $\phi$ satisfying \eqref{phi=0} and \eqref{det=0}.
 \item In Theorem \ref{th:osc},  we prove Arnold's conjecture. More precisely, we show that the leading term of the asymptotic expansion of \eqref{Oscil} is determined by the height $h(\phi)$ of $\phi$ and the dimension of the principal face of the Newton polyhedron  in this adapted coordinate system.       Our asymptotic decay rates are sharp.
 \item In Corollary \ref{cor:Karp}, we show that our asymptotic estimate on \eqref{Oscil}  is stable under sufficiently small analytic perturbation of $\phi.$ This result is analogous to Karpushkin's result \cite{Kar} in two dimensions.
\end{itemize}

 The problem of finding the exact decay rate of the Fourier transform of surface-carried measures is closely related to the $L^p$-boundedness problem of maximal operators.   As before, let $d\sigma$ denote the surface measure on a hypersurface $S$ in $\mathbb{R}^{n+1}$ and let $ \rho \in C_0^{\infty}(S)$ be non-negative.  We define the averaging operator
 \begin{equation*}
 \mathcal{A}_tf(y)\colonequals \int_Sf(y-tx)\rho(x)d\sigma(x),\,\,\,\,t>0,\,\,\,\,f\in C^{\infty}_0(\mathbb{R}^{n+1}) .
 \end{equation*}
 The  maximal operator associated with the hypersurface $S$ is given by
 \begin{equation*}
 \mathcal{M}f(y)\colonequals \sup_{t>0}|\mathcal{A}_tf(y)|, \,\,\,\,f\in C^{\infty}_0(\mathbb{R}^{n+1}) .
 \end{equation*}
 We say that the maximal operator $\mathcal{M}$ is bounded on $L^p(\mathbb{R}^{n+1})$ if there exists a positive constant $C_p$ such that
 \begin{equation*}
 ||\mathcal{M}f||_{L^p(\mathbb{R}^{n+1})}\leq C_p||f||_{L^p(\mathbb{R}^{n+1})},\,\,\,\,\forall\,f\in C^{\infty}_0(\mathbb{R}^{n+1}).
 \end{equation*}
 The fundamental question is for which $p\geq 1$ the maximal operator $\mathcal{M}$ is bounded on $L^p(\mathbb{R}^{n+1})$.
 Since $\rho$ has a compact support,  $\mathcal{M}$ is always bounded on $L^{\infty}(\mathbb{R}^{n+1})$. Hence, the question is to determine 
 \begin{equation*}
 p(S)\colonequals \inf\{p\geq 1:
 \mathcal{M}\,\,\text{is bounded on}\,\,L^p(\mathbb{R}^{n+1}) \text{ for all } 0\leq \rho\in C_0^\infty(S)\}.
 \end{equation*}
The number $p(S)$ is called  the \emph{boundedness exponent} of the maximal operator $\mathcal{M}$ associated with the hypersurface $S$.

    Note that, unlike linear transformations, translations do not commute with
dilations, which is why Euclidean motions are not admissible coordinate changes for the
study of the maximal operator $\mathcal{M}$. Hence, as it is done in \cite{I.K.M,P1,P2},   we require the  transversality assumption on $S,$ i.e.,
for every $x\in S,$ the affine tangent plane $x+T_xS $ to $S$ through $x$ does not pass the origin in $\mathbb{R}^{n+1}.$ 
This  transversality assumption is natural in this context. Indeed, if this condition is not met, then the behavior of the maximal operator may change significantly.
For studies without the transversality condition, we refer to \cite{Zim} for hypersurfaces in $\mathbb{R}^3$ and \cite{Lee.Oh} for the multidimensional case.

  Using a partition of unity argument, we can assume that $\rho$ is concentrated in a sufficiently small neighborhood of a fixed point $x^0\in S.$  
   The transversality assumption allows
us to find a linear change of coordinates in $\mathbb{R}^{n+1}$ so that in the new coordinates
$S$ can locally be represented as the graph of a smooth function. The norm of
$\mathcal{M}$ when acting on $L^p$ is invariant under such a linear change of coordinates.
More precisely, after applying a suitable linear change of coordinates to $\mathbb{R}^{n+1},$ we may
assume that $x^0 = (0,...,0,1)\in \mathbb{R}^{n+1},$  and   $S$ is given as the
graph
\begin{equation}\label{Smax}
S=\{(x, 1+\phi(x))\in \mathbb{R}^{n+1}:\,x\in U\},
\end{equation}
 where $U\subset \mathbb{R}^{n}$ is a small neighborhood  of the origin and  $\phi$ is a smooth function   satisfying $\phi(0)=0,$ $\nabla \phi (0)=0.$
Then the maximal  operator  can be written as
 \begin{equation*}
 \mathcal{M}f(y)= \sup_{t>0}\left| \int_{\mathbb{R}^n}f(y_1-tx_1,..., y_n-tx_n, y_{n+1}-t(1+\phi(x)))\eta(x)dx\right|, \,\,\,\,f\in C^{\infty}_0(\mathbb{R}^{n+1}),
 \end{equation*}
 where $\eta(x)\colonequals \rho(x_1,...,x_n,1+\phi(x))\sqrt{1+|\nabla \phi(x)|^2}\in C^{\infty}_0(U).$

 The first fundamental result on the $L^p$-boundedness of the maximal operator was due to Stein \cite{Ball}, who proved that if $S$ is the unit sphere centered at the origin in $\mathbb{R}^{n+1}$ with $n \geq 2$, then the corresponding spherical maximal operator is bounded on $L^p$ for every $p > 1 + {1}/{n}$.
 A two-dimensional version of this result was proven by Bourgain \cite{Bo}.
 The non-vanishing Gaussian curvature of the unit sphere plays a crucial role in these articles.  A more general result \cite{Gr} is  that, if $S$   has non-zero $k\geq 2$ principal curvatures at each point of $\mathrm{supp}(\rho)$, then the corresponding maximal operator is bounded on $L^p$ for every $p>1+1/k.$ Later, a similar result was obtained for the more subtle  case $k=1$ by Sogge in \cite{Sogge}.  Sogge and Stein \cite{Sog.St} showed that, if the Gaussian curvature of $S$ does not vanish to infinite order  at    any point of $S,$ then $\mathcal{M}$ is bounded on $L^p$ for a sufficiently large  $p$. Moreover, if $S$ is a smooth finite type hypersurface, then the boundedness exponent $p(S)$ is finite, see \cite{I.U}.  However, the exact value of  $p(S)$ is not determined in these articles, except for the case of $n=1$.  Iosevich \cite{Io} studied the maximal operator associated with finite type curves in $\mathbb{R}^2$ and determined the exact value of $p(S)$.
 Iosevich and Sawyer \cite{Io.Saw} showed that, if the maximal operator $\mathcal{M}$ is bounded on $L^p, $ then
 $
  d(x,H)^{-1/p}\in L^{1}_{\mathrm{loc}}(S),
    $
  where $H$ is any hyperplane not passing through the origin and $d(x,H)$ denotes the distance from $x\in S$ to $H$. It was conjectured in \cite{Io.Saw}  that  for $p>2$  the above condition is necessary and sufficient for the boundedness of the maximal operator $\mathcal{M}$ on $L^p.$
 An essentially complete answer to the  $L^p$-boundedness of $\mathcal{M}$ has been given by  Ikromov, M\"uller et al. \cite{I.K.M, P1,P2} for  hypersurfaces $S$ of finite type in $\mathbb{R}^3.$  In these studies, the boundedness exponent $p(S)$ is determined by the  height $h(\phi)$  of the function $\phi$ when $S$ has the form in  \eqref{Smax}.
In higher dimensions, there are  studies for convex hypersurfaces: Using Schulz's decomposition  of convex functions in \cite{Sch},  Iosevich and Sawyer \cite{I.S} proved sharp $L^p$-estimates for  smooth convex hypersurfaces $S$ of finite line type for $p>2.$  This result has been extended in \cite{I.U}  for analytic convex hypersurfaces $S$ of finite type, which are not necessarily  finite line type. However, in dimensions $n\geq 3,$ the $L^p$-boundedness
 of maximal operators associated with non-convex hypersurfaces $S$ whose
Gaussian curvature vanishes at some points is still  open. In this paper, we aim to contribute the  three-dimensional case $n=3$.

As we mentioned before, asymptotic behavior of the Fourier transform of  surface-carried measures is intimately connected with the   $L^p$-boundedness  of  maximal  operators:  Greenleaf \cite{Gr}  proved that, if
\begin{equation}\label{O}|\widehat{\rho d\sigma}(\xi)|\le C(1+|\xi|)^{-q}
\end{equation}
for some $C>0$ and  $q>1/2,$ then the maximal operator is bounded on $L^p$ whenever $p>1+1/(2q).$ However, this result is not optimal in general.  Stein, for $q=1/2,$ and later Iosevich  and  Sawyer \cite{I.S}, for $0<q<1/2,$ conjectured that if \eqref{O} holds for some $0<q\leq 1/2,$ then  $\mathcal{M}$ is bounded for $p>1/q.$ This conjecture was proven for curves in $\mathbb{R}^2$ in \cite{Io}, for hypersurfaces of fine type in $\mathbb{R}^3$ in \cite{I.K.M},  and for convex hypersurfaces in any dimension in \cite{I.U}.  This conjecture with $q=1/2$ has been recently proven by Oh \cite{Oh} for all smooth hypersurfaces.

 In this paper, we also study the $L^p$-boundedness of the maximal operators associated with hypersurfaces $S$ in $\mathbb{R}^4$ whose Gaussian curvature is identically zero and  given as graphs of  polynomial functions.
This means that we consider hypersurfaces $S\subset \mathbb{R}^4$  given as the
graph
\begin{equation*}
S=\{(x_1,x_2,x_3, 1+\phi(x_1,x_2,x_3))\in \mathbb{R}^{4}:\,(x_1,x_2,x_3)\in U\},
\end{equation*}
 where $U\subset \mathbb{R}^3$ is a neighborhood of the origin and $\phi$ is a  polynomial function satisfying \eqref{phi=0} and \eqref{det=0}.
 Our main results concerning the maximal operator $\mathcal{M}$ are the following:
 \begin{itemize}
 \item In Theorem \ref{th:Max}, we prove that the maximal operator $\mathcal{M}$ is bounded on $L^p(\mathbb{R}^4)$ for $p>\max\{h(\phi), 2\},$ where $h(\phi)$ is the height of $\phi.$
 \item In Proposition \ref{prop:max}, based on \cite{Io.Saw} we prove a necessary condition: if the maximal operator $\mathcal{M}$ with  $\rho(x^0)>0$ is bounded on $L^p(\mathbb{R}^4)$ for $p>1$, then we necessarily have $p>h(\phi)$. In the case $h(\phi)\geq 2,$  Theorem \ref{th:Max} and  Proposition \ref{prop:max} yields (see Corollary \ref{cor:max}) that the maximal operator $\mathcal{M}$ with $\rho(x^0)>0$ is bounded on $L^p(\mathbb{R}^4)$ if and only if $p>h(\phi)$, and so $p(S)=h(\phi).$
 In addition, we obtain confirmation of the Iosevich-Sawyer  and  Stein-Iosevich-Sawyer conjectures.

 \item In Theorem \ref{th:h<2}, we consider the case  $h(\phi)<2$:  by assuming  that all the  principal curvatures of $S$  are zero at $x^0=(0,0,0,1)\in S$ (i.e., $D^2\phi(0)=0$), we  prove that the maximal operator $\mathcal{M}$ with $\rho(x^0)>0$ is bounded on $L^p(\mathbb{R}^4)$ if and only if $p>h(\phi)$, and so $p(S)=h(\phi)$. When two principal curvatures of $S$  are non-zero at $x^0=(0,0,0,1)\in S$,  i.e., $\mathrm{rank}(D^2\phi(0))=2$, using \cite{Gr} we show that  the maximal operator $\mathcal{M}$ is bounded on $L^p(\mathbb{R}^4)$ for $p>3/2.$ The case where only one principal curvature is non-zero at $x^0=(0,0,0,1)\in S$ remains unsolved, and the exact value of the boundedness exponent $p(S)$ is unknown, even for hypersurfaces in $\mathbb{R}^3,$ see \cite{P2} for an appropriate geometric conjecture 
 for this case. 
\end{itemize}

  In analogy with
Arnold's notion of the \emph{singularity index}   \cite{Ar,agvMN82},  one can define the so-called \textit{(uniform) oscillation index} and  \textit{(uniform) contact index} at  any point $x^0$ of  a smooth hypersurface $S$ (see Section 6 for their definitions).  These indexes are related to  the  Fourier transform of surface-carried measures and the maximal operator associated with $S$. In this paper, we also consider these indexes for a hypersurface $S$ in $\mathbb{R}^4$ with zero Gaussian curvature and given as  a graph of  a polynomial function $\phi$.
 Our result concerning these indexes is given in Theorem \ref{Main}, where we show that all of these indexes equal to  $1/h(\phi).$

The organization of this paper is as follows. In Section 2, we define the Newton polyhedron of  smooth functions and related notions such as adapted coordinate system and height. In Section 3,  we study some properties of polynomial functions with vanishing Hessian determinant and prove that there are   adapted coordinate systems  for them.  Our  results concerning   oscillatory integrals and their proofs  are given in Section 4.  Section 5 is devoted  to the $L^p$-boundedness of maximal operators.  In Section 6,  we define (uniform) oscillation and  (uniform) contact indexes of hypersurfaces and determine their exact value. Finally, we discuss possible extensions of our results in Section 7.
\section{The Newton polyhedron}
Let $\phi$ be a smooth real-valued function defined on an open neighborhood $U$ of the origin in $\mathbb{R}^n$ with $\phi(0)=0$ and $\nabla \phi(0)=0.$ We consider the associated Taylor series at the origin
$$\phi(x)\approx \sum_{\alpha \in \mathbb{N}_0^n}c_{\alpha}x^{\alpha},$$
where $\mathbb{N}_0\colonequals \mathbb{N}\cup \{0\}$ and   $ c_{\alpha}\colonequals  \frac{1}{\alpha_1!...\alpha_n!}\frac{\partial^{|\alpha|}\phi(0)}{\partial x_1^{\alpha_1}...\partial x_n^{\alpha_n}}$ for the multi-index $\alpha=(\alpha_1,...,\alpha_n)\in \mathbb{N}_0$ and $|\alpha|\colonequals \alpha_1+...+\alpha_n$. 
The set
$$\mathcal{T}(\phi)\colonequals \{\alpha \in \mathbb{N}_0^n\setminus \{0\}:\,c_{\alpha}\neq 0\}  $$
is called the \emph{Taylor support} of $\phi$ at the origin. We assume $\mathcal{T}(\phi)$ is not empty, i.e., the function $\phi$ is \emph{finite type} at the origin. The \emph{Newton polyhedron} $\mathcal{N}(\phi)$ of $\phi$ at the origin is defined to be the convex hull of the union $$\bigcup_{\alpha \in \mathcal{T}(\phi) } (\alpha+\mathbb{R}^n_+),$$
where $\mathbb{R}_+\colonequals \{x\in \mathbb{R}:\,\,x\geq 0\}.$ The \emph{Newton diagram } $\mathcal{N}_d(\phi)$ in the sense of Varchenko \cite{Var} is the union of all compact faces of the Newton polyhedron. 
 Let $(d,...,d)\in \mathbb{R}^n$ be the coordinate of the intersection of the line $x_1=...=x_n$ and the boundary of the Newton polyhedron. The number $d=d(\phi)$ is called the \emph{distance}  between the Newton polyhedron and the origin.

The principal face $\pi (\phi)$ of the Newton polyhedron of $\phi$ is the face of the minimal dimension containing the point $(d,...,d)\in \mathbb{R}^n.$ We call the series
$$\phi_{\text{pr}}(x)\colonequals \sum_{\alpha\in \pi (\phi) } c_{\alpha} x^{\alpha}$$
 the \emph{principal part} of $\phi.$ In case $\pi (\phi)$ is compact, $\phi_{\text{pr}}(x)$ is a 
  polynomial; otherwise, we shall consider $\phi_{\text{pr}}(x)$ as a formal power series.

We note that the distance between the Newton polyhedron and the origin depends on the chosen \emph{local coordinate system} at the origin. Here by a local coordinate system we mean any smooth diffeomorphic map of a neighborhood of the origin into itself such that the origin is a fixed point.  The \emph{height} of the smooth function $\phi$ is defined by
\begin{equation}\label{height}h(\phi)\colonequals \sup\{d_x\},
\end{equation}
where the supremum is taken over all local smooth coordinates systems $x$ at the origin and $d_x$ is the distance between the Newton polyhedron and the origin in the  coordinate system $x$. We say a given smooth local coordinate system is \emph{adapted} to $\phi$ if $h(\phi)=d_x.$
If the supremum in \eqref{height} is attained  by a linear transformation of variables, then the  coordinate system obtained by applying this  linear transformation of variables  is said to be linearly adapted to $\phi$. 

In two dimensions, Varchenko \cite{Var} gave sufficient conditions for the adaptness of a given coordinate system and proved the existence of an adapted coordinate system for analytic functions without multiple components. This result is extended in \cite{Adap} for arbitrary real-analytic functions and arbitrary smooth functions of finite type.  In three dimensions, in general, no analogs of such coordinate
systems exist because an counterexample was constructed  by  Varchenko \cite{Var}. Adapted coordinate systems exist for  smooth convex functions in any dimension and it is obtained by applying linear change of variables to the original orthogonal coordinate system, see \cite{Sch, I.U}.

In the following  we state some results of \cite{Adap} concerning the adaptness of a given coordinate system for a smooth function in two dimensions. If the principal face of the Newton polyhedron $\mathcal{N}(\phi)$  is a compact edge in two dimensions, then it lies on a unique line $\kappa_1 t_1+\kappa_2 t_2=1$ with $\kappa_1,\kappa_2>0.$  The principal part $\phi_{\text{pr}}$ of $\phi$ is $\kappa-$homogeneous of degree one, i.e.,
$$\phi_{\text{pr}}(t^{\kappa_1}x_1,t^{\kappa_2}x_2)=t\phi_{\text{pr}}(x_1,x_2),\,\,\,\, \forall\,t>0, \,\,\,\forall\, x\in U. $$
$\kappa\colonequals(\kappa_1,\kappa_2)$ is called the principal weight and $d_h(\phi_{\text{pr}})\colonequals 1/(\kappa_1+\kappa_2)$ is called the homogeneous distance. We denote by $m(\phi_{\text{pr}})\colonequals \text{ord}_{\mathbb{S}^1}\phi_{\text{pr}}$
the maximal order of vanishing of $\phi_{\text{pr}}$ along the unit circle $\mathbb{S}^1$ centered at the origin.
\begin{proposition}\label{pr:a} 
Let $\phi$ be a smooth real-valued function defined on an open neighborhood $U$ of the origin in $\mathbb{R}^2$ with $\phi(0)=0$ and $\nabla \phi(0)=0.$ The given coordinate system  $x$ is adapted to $\phi$ if and only if one of the following conditions is satisfied:
\begin{itemize}
\item[(i)] The principal face $\pi (\phi)$ of the Newton polyhedron is a compact edge, and $m(\phi_{\mathrm{pr}})\leq d(\phi).$
\item[(ii)] $\pi (\phi)$ is a vertex.
\item[(iii)] $\pi (\phi)$ is an unbounded edge.
\end{itemize}
Moreover, in case of $(i)$ we have $h(\phi)=
h(\phi_{\mathrm{pr}})=
d_h(\phi_{\mathrm{pr}}).$
\end{proposition}
The proof of Proposition \ref{pr:a} can be found in \cite{Adap}.

\section{Adapted coordinate system}

In this section we prove that there is an adapted coordinate system for any polynomial function whose Hessian determinant is identically zero.

Let $\phi:\mathbb{R}^n\to \mathbb{R}$ be a polynomial function. Let $\det{(D^2\phi(x))}$ denote the Hessian determinant
$$\det{(D^2\phi(x))}\colonequals \begin{vmatrix}
\partial^2_{x_1x_1}\phi(x) &...&\partial^2_{x_1x_n}\phi (x)\\\vdots  & & \vdots\\
\partial^2_{x_nx_1}\phi(x) &...&\partial^2_{x_nx_n}\phi (x)
\end{vmatrix},\,\,\,\,x\in \mathbb{R}^n.$$
When $\phi$ is  a homogeneous polynomial, Hesse \cite{H1, H2} claimed in 1851 that the equality
$$\det{(D^2\phi(x))}=0$$
 holds for all $x\in \mathbb{R}^n$ if and only if, after a suitable linear  transformation of  variables,  $\phi$ depends on at most $n-1$ variables. The \emph{if} part in this claim is trivial and the \emph{only if} part is not obvious. One can easily show  that this claim holds true for  quadratic forms. It is also true for ternary and quaternary cubic forms, see \cite{Pasch}.  However, in 1876, Gordan and Noether \cite{GN, Los, J.B} showed that Hesse's proof in \cite{H1, H2} is not correct in general. More precisely, they showed that Hesse's claim holds for $n\leq 4$ and constructed counterexamples  for $n\geq 5.$ 
For non-homogeneous polynomials  the following result has been obtained in \cite[Theorem 3.3]{Bondt}:
\begin{theorem}\label{Bon}
Let  $\phi:\mathbb{R}^3\to \mathbb{R}$ be a  polynomial  such that $\phi(0)=0$ and $\nabla \phi(0)=0.$ Assume  $$\det{(D^2\phi(x))}=0$$
holds   for all $x=(x_1,x_2,x_3)\in \mathbb{R}^3$.  Then
then there exists an  invertible matrix $A\in \mathbb{R}^{3\times 3}$ such that $\phi(Ax)$ either depends on at most two  variable $x_1$ and $x_2$ or has the form
\begin{equation}\label{form}
\phi(Ax)=Q_1(x_1)+Q_2(x_1)x_2+Q_3(x_1)x_3,
\end{equation}
where $Q_1,Q_2,Q_3:\mathbb{R}\to \mathbb{R}$ are polynomial functions.
\end{theorem}
We use this algebraic result to prove the main theorem of this section:
\begin{theorem}\label{th:adap}
Let  $\phi:\mathbb{R}^3\to \mathbb{R}$ be a non-trivial polynomial function such that $\phi(0)=0$ and $\nabla \phi(0)=0.$  Assume $$\det{(D^2\phi(x))}=0$$
holds  for all $x=(x_1,x_2,x_3)\in \mathbb{R}^3$.   Then there exists an adapted coordinate system to $\phi$ in $\mathbb{R}^3.$ Moreover, let  $A\in \mathbb{R}^{3\times 3}$ be the matrix given in  Theorem \ref{Bon}. If $\phi(Ax)$ only depends on one variable, then the given coordinate  system $x$ is adapted. If it depends on two variables, then an adapted coordinate system can be found by either $y_1 = x_1$, $y_2 = x_2 - \psi(x_1)$ or $y_1 = x_2$, $y_2 = x_1 - \psi(x_2)$, where $\psi$ represents a real-valued analytic function. If $\phi(Ax)$ has a form as in \eqref{form}, then an adapted coordinate system can be obtained   by applying a linear transformation of variables.
\end{theorem}
\begin{proof}
We divide the proof into two parts.

\textbf{Part 1:} Assume $ \phi(Ax)$ depends on at most two variables $x_1$ and $x_2$.

$(a)$ Let $ \phi(Ax)$ depend on only one variable $x_1.$ Then we can write
$$ \phi(Ax)=x_1^{\nu} Q(x_1),$$
where $ Q$ is a polynomial function with $ Q(0)\neq 0$ and  $\nu\geq 2$ is a positive integer number.  Then the Newton polyhedron of $ \phi(Ax)$ is
$$\{(t_1, t_2, t_3)\in \mathbb{R}^3_+:\, t_1\geq \nu\}$$
and the distance $d_x=\nu.$
We claim that the given  coordinate system  $x$ is adapted. To prove that we consider any smooth local coordinate system $y$ at the origin given by
 \begin{align}\label{new.coor}
 &x_1=\varphi_1(y_1,y_2,y_3), \nonumber\\
 &x_2=\varphi_2(y_1,y_2,y_3), \nonumber \\ & x_3=\varphi_3(y_1,y_2,y_3).
 \end{align}
 Let $d_y$ denote the distance between the Newton polyhedron and the origin in the coordinate system $y.$ By definition, the coordinate system $x$ is adapted if $d_y\leq d_x$ for any smooth local coordinate system $y.$ We will prove that this inequality holds.
Since the Jacobian  of $\varphi$ must be non-zero at the origin, by flipping coordinates $y_1,$ $y_2,$ and $y_3,$ if necessary, we may assume without loss of generality that $\partial_{y_i}\varphi_i(0,0,0)\neq 0 $ for all $i\in \{1,2,3\}.$  
 In this new coordinate system $\phi (Ax)$ equals
 $$\Phi(y)=\varphi_1^{\nu}(y) Q(\varphi_1^{\nu}(y)).$$
Since $\partial_{y_1}\varphi_1(0)\neq 0$ and $Q(0)\neq 0,$ by product rule, $(\nu,0,0)$ belongs to the Newton polyhedron of $\Phi(y).$ This means  the distance $d_y$ can not be greater than $d_x=\nu.$ Therefore, the given coordinate system $x$ is adapted.

$(b)$ Let $ \phi(Ax)$ depend on two variables   $x_1$ and $x_2.$ Then the problem of finding an adapted coordinate system in $\mathbb{R}^3$ reduces to find  an adapted coordinate system in  $\mathbb{R}^2.$    If the given coordinate system $x$ in $\mathbb{R}^2$ is not adapted, then \cite[Theorem 4.2]{Adap} shows that an adapted coordinate system in $\mathbb{R}^2$ can be obtained by either \begin{align*}
y_1&\colonequals x_1,\\
y_2& \colonequals x_2-\psi(x_1)
\end{align*}
or
\begin{align*}
y_1&\colonequals x_2,\\
y_2&\colonequals x_1-\psi(x_2),
\end{align*}
where  $\psi$ is a real analytic function.

\bigskip
\textbf{Part 2:} Assume  $ \phi(Ax)$ has the form
\begin{equation*}
\phi(Ax)=Q_1(x_1)+Q_2(x_1)x_2+Q_3(x_1)x_3,
\end{equation*}
where $Q_1, Q_2, Q_3:\mathbb{R}\to \mathbb{R}$ are polynomial functions.
Let $\nu_1,\nu_2,\nu_3\in \mathbb{N}$   denote the multiplicities of $x_1=0$ as the root  of $Q_1(x_1),$ $Q_2(x_1),$ and $Q_3(x_1),$ respectively.  Since $\phi(0)=0$ and $\nabla \phi(0)=0,$ we have  $\nu_1\geq 2,$ $\nu_2\geq 1,
$ and $\nu_3\geq 1.$

If  $Q_2(x_1)$ or  $Q_3(x_1)$ is identically zero, then $\phi(Ax)$ depends on at most two variables and so Part 1  implies the proof. Hence, we assume that both $Q_2(x_1)$ and $Q_3(x_1)$ are not identically zero. $Q_1(x_1)$ can be identically zero. In this case, we formally put $\nu_1=\infty$.

 We can write
$$Q_i(x_1)= x_1^{\nu_j}\tilde Q_i(x_1),\,\,\,i\in\{1,2,3\} $$
where $\tilde Q_i$ are polynomial functions. According to our assumption we have $\tilde Q_i(0)\neq 0$    for $i\in \{2,3\},$ and $\tilde Q_1(0)\neq 0$  if  $Q_1$ is not identically zero.
The Newton polyhedron of $\phi(Ax)$ is the convex hull of
$$\{(\nu_2,1,0)+\mathbb{R}_+^3\}\cup \{(\nu_1,0,0)+\mathbb{R}_+^3\} \cup \{(\nu_3,0,1)+\mathbb{R}_+^3\}.$$
Without loss of generality we can assume $\nu_3\geq \nu_2.$  Then  $\min\{\nu_1,\nu_2,\nu_3\}$  can be $\nu_1$ or $\nu_2.$
Hence, we consider two different cases:

\textbf{Case 1: } Let $\nu_1= \min\{\nu_1,\nu_2,\nu_3\}.$
 In this case the Newton polyhedron of $\phi(Ax)$ is
$$\{(t_1, t_2, t_3)\in \mathbb{R}^3_{+}: \,t_1\geq \nu_1 \}$$
and the distance $d_x=\nu_1.$ By definition, the coordinate system $x$ is adapted if $d_y\leq d_x$ for any smooth local coordinate system $y.$ We will prove that this inequality holds.
  We consider any smooth  local coordinate system $y$ as in \eqref{new.coor}.  In these new coordinates $\phi (Ax)$ equals
 \begin{align*}
 \Phi(y)
 =\varphi_1^{\nu_1}(y)\tilde Q_1(\varphi_1(y))+\varphi_1^{\nu_2}(y) \varphi_2(y)\tilde Q_2(\varphi_1(y))+\varphi_1^{\nu_3}(y) \varphi_3(y)\tilde Q_3(\varphi_1(y)).
 \end{align*}
 Using $\partial_{y_1}\varphi_1(0)\neq 0$ and the product rule, one can show that $(\nu_1,0,0)$ belongs to the Newton polyhedron of $\Phi(y).$ Hence, the distance $d_y$ can not be greater than $d_x.$

 \textbf{Case 2:} Let $\nu_2= \min\{\nu_1,\nu_2,\nu_3\}.$  If $\nu_1=\nu_2,$ then Case 1 implies the proof. Hence, we assume $\nu_1>\nu_2.$ 
 We introduce the following notations
 \begin{align*}
 c_1&\colonequals \begin{cases} \tilde Q_1(0)\quad &\mbox{if} \quad \nu_1=\nu_2+1\\
 0 \quad & \mbox{if} \quad \nu_1>\nu_2+1,
  \end{cases}\\
 c_2&\colonequals \tilde Q_2(0)\neq 0, \\
  c_3&\colonequals \begin{cases} \tilde Q_3(0)\quad &\mbox{if} \quad \nu_3=\nu_2\\
 0 \quad &\mbox{if} \quad \nu_3>\nu_2.
  \end{cases}
   \end{align*}
With these notations we can write
\begin{equation*}
\phi(Ax)=x_1^{\nu_2}(c_1x_1+c_2x_2+c_3x_3)+x_1^{\nu_1}(\tilde Q_1(x_1)-c_1)+x_1^{\nu_2}x_2(\tilde Q_2(x_1)-c_2)+
x_1^{\nu_3}x_3(\tilde Q_3(x_1)-c_3).
\end{equation*}
 We change the variables
\begin{align*}
&z_1=x_1,\\
&z_2=c_1x_1+c_2x_2+c_3x_3\\
&z_3=x_3.
\end{align*}
Let $B$ denote  the matrix  product of $A\in \mathbb{R}^{3\times 3}$ and $\begin{pmatrix}
1&0&0\\
-c_1/c_2&1/c_2&-c_3/c_2\\
0&0&1
\end{pmatrix}. $
After changing the variables we get
\begin{align*}
\phi(Bz)=&z_1^{\nu_2}z_2+z_1^{\nu_1}(\tilde Q_1(z_1)-c_1)-\frac{c_1}{c_2}z_1^{\nu_2+1}(\tilde Q_2(z_1)-c_2)\\ &+\frac{z_1^{\nu_2}z_2}{c_2}(\tilde Q_2(z_1)-c_2)+\left(z_1^{\nu_3}(\tilde Q_3(z_1)-c_3)-\frac{c_3z_1^{\nu_2} }{c_2}(\tilde Q_2(z_1)-c_2)\right)z_3.
\end{align*}
Since $\nu_2\geq 1,$ we have $d_z=\nu_2.$ We claim that the coordinate system $z$ is adapted. To prove that we consider any smooth local coordinate system $y$ as in  \eqref{new.coor}. In this new coordinates, $\phi(Bz)$  equals
\begin{align*}
\Phi(y)=&
\varphi_1^{\nu_2}(y)\varphi_2(y)+\varphi_1^{\nu_1}(y)(\tilde Q_1(\varphi_1(y))-c_1)\\
&-\frac{c_1}{c_2}\varphi_1^{\nu_2+1}(y)(\tilde Q_2(\varphi_1(y))-c_2) +\frac{\varphi_1^{\nu_2}(y)\varphi_2(y)}{c_2}(\tilde Q_2(\varphi_1(y))-c_2)\\
&+\left(\varphi_1^{\nu_3}(y)(\tilde Q_3(\varphi_1(y))-c_3)-\frac{c_3\varphi_1^{\nu_2}(y) }{c_2}(\tilde Q_2(\varphi_1(y))-c_2)\right)\varphi_3(y).
\end{align*}
Note that  every monomial of 
\begin{align*}
z_1^{\nu_1}(\tilde Q_1(z_1)-c_1)&-\frac{c_1}{c_2}z_1^{\nu_2+1}(\tilde Q_2(z_1)-c_2) +\frac{z_1^{\nu_2}z_2}{c_2}(\tilde Q_2(z_1)-c_2)\\ &  +\left(z_1^{\nu_3}(\tilde Q_3(z_1)-c_3)-\frac{c_3z_1^{\nu_2} }{c_2}(\tilde Q_2(z_1)-c_2)\right)z_3
\end{align*}
 has a degree at least  $\nu_2+2.$ Therefore,
using $\partial_{y_1}\varphi_1(0)\neq 0,$ $\partial_{y_2}\varphi_2(0)\neq 0,$ and the product rule, one can check that $(\nu_2,1,0)$ belongs to the Newton polyhedron of $\Phi(y).$ Hence,  $d_y$ can not be greater than $d_z=\nu_2.$

\end{proof}

The proof of Theorem \ref{th:adap} yields the following corollary:
\begin{corollary}\label{cor:obvious}
Let  $\phi:\mathbb{R}^3\to \mathbb{R}$ be a  polynomial function  satisfying the conditions of Theorem   \ref{th:adap}. Then, after a linear change of variables,  either  $\phi$ depends on one or two variables  or there exists a linearly adapted coordinate system.
\end{corollary}

\begin{corollary}\label{esseen}
Let  $\phi:\mathbb{R}^3\to \mathbb{R}$ be a  polynomial function  satisfying the conditions of Theorem   \ref{th:adap}.  Then, after  a
local analytic  change of variables,   $\phi$ depends on at most two variables.
\end{corollary}
The proof of Corollary \ref{cor:obvious}  follows trivially from Theorem \ref{th:adap}.  The proof of Corollary \ref{esseen} is not obvious, hence in the following we provide its proof.
\begin{proof}[{Proof of Corollary \ref{esseen}}] Theorem \ref{Bon} shows that there is an invertible matrix $A$ such that $\phi(Ax)$ depends on at most two variables or has the form in \eqref{form}. If $\phi(Ax)$ depends on at most two variables, then the corollary is proven. It remains  to consider the case  where  $\phi(Ax)$ has the form in  \eqref{form}. As in the proof of Theorem \ref{th:adap} we can consider two cases. In Case 1, we  change the variables
\begin{align*}
y_1&=x_1|\tilde Q_1(x_1)+x_1^{\nu_2-\nu_1}x_2 \tilde{Q}_2(x_1)+x_1^{\nu_3-\nu_1}x_3\tilde{Q}_3(x_1)|^\frac1{\nu_1},\\
y_2&=x_2,\\
y_3&=x_3.
\end{align*}
One can check that this change of variables is analytic in a sufficiently small neighborhood of the origin.
 Then, in this new coordinate system $y,$ $\phi(Ax)$ equals $\mathrm{sign}(\tilde Q_1(0))y_1^{\nu_1}.$ So, it depends on one variable $y_1$.

As shown in the proof of Theorem \ref{th:adap}, in Case 2 there is an invertible matrix $B$ such that
\begin{align*}
\phi(Bz)=&z_1^{\nu_2}z_2+z_1^{\nu_1}(\tilde Q_1(z_1)-c_1)-\frac{c_1}{c_2}z_1^{\nu_2+1}(\tilde Q_2(z_1)-c_2)\\ &+\frac{z_1^{\nu_2}z_2}{c_2}(\tilde Q_2(z_1)-c_2)+\left(z_1^{\nu_3}(\tilde Q_3(z_1)-c_3)-\frac{c_3z_1^{\nu_2} }{c_2}(\tilde Q_2(z_1)-c_2)\right)z_3.
\end{align*}
 It is sufficient to use the following change of variables
 \begin{align*}
 y_1&=z_1,\\
  y_2&= z_2+z_1^{\nu_1-\nu_2}(\tilde Q_1(z_1)-c_1)-\frac{c_1}{c_2}z_1^{\nu_2+1-\nu_1}(\tilde Q_2(z_1)-c_2) \\ &\quad +\frac{z_2}{c_2}(\tilde Q_2(z_1)-c_2)+\left(z_1^{\nu_3-\nu_2}(\tilde Q_3(z_1)-c_3)-\frac{c_3 }{c_2}(\tilde Q_2(z_1)-c_2)\right)z_3,\\
  y_3&=z_3.
 \end{align*}
Obviously, this change of variables is  analytic  in a sufficiently small neighborhood of the origin. In this new coordinate system $y$, $\phi(Bz)$ equals $ y_1^{\nu_2} y_2.$
So, it depends on two variable $y_1$ and $y_2$.

  \end{proof}


\section{Estimates for oscillatory integrals}

In this section we study the asymptotic behavior of the oscillatory integral \eqref{Oscil}.

Let $\phi:\mathbb{R}^3\to \mathbb{R}$ and $A\in \mathbb{R}^{3\times 3}$ be a polynomial function and an invertible matrix as in Theorem \ref{Bon}, respectively. Then $ \phi(Ax)$ either depends on at most two variables $x_1$ and $x_2$ or has the form as in \eqref{form}. We define the index $\nu(\phi)\in \{0,1\}:$ When
$ \phi(Ax)$  depends on at most two variables $x_1$ and $x_2,$    we have shown in Theorem \ref{th:adap} that there is an adapted coordinate system $y$ in $\mathbb{R}^2.$ If the principal face $\pi(\phi^a)$ of $\phi,$ when expressed by the function $\phi^a$ in the adapted coordinate system $y,$  is a vertex, and if $h(\phi)\geq 2,$ then we put $\nu(\phi)\colonequals 1;$ otherwise we put $\nu(\phi)\colonequals 0.$
 When $\phi(Ax) $ has the form as in \eqref{form}, we  put $\nu(\phi)\colonequals 0$.

 The index $\nu(\phi)$ is defined in \cite{Ik.Mu} for functions depending only two variables, where it is called \emph{Varchenko's exponent}. When $h(\phi)\geq 2,$ $\nu(\phi)$ coincides with the index in \cite[Theorem 0.4]{Var}, where one can find how $\nu(\phi)$ is related to the dimension of the principal face of the Newton polyhedron in an adapted coordinate system.

 With this definition of $\nu(\phi)$ we can now state the main results of this section:
\begin{theorem}\label{th:osc}
Let  $\phi:\mathbb{R}^3\to \mathbb{R}$ be a  polynomial function such that $\phi(0)=0$ and $\nabla \phi(0)=0.$  Assume $$\det{(D^2\phi(x))}=0$$
holds  for all $x=(x_1,x_2,x_3)\in \mathbb{R}^3$. Let $h=h(\phi)$ be the height of $\phi$ defined in \eqref{height} and $\nu=\nu(\phi)\in \{0,1\}$ be defined as above.  Then there exists a neighborhood $U\subset \mathbb{R}^3$ of the origin and a constant $C>0$ such that, for every $ \eta \in C^{\infty}_0(U)$ and $\xi=(\xi_1,\xi_2,\xi_3,\xi_4)\in \mathbb{R}^4,$ the following estimate holds:
\begin{equation}\label{main}
\left|\int_{\mathbb{R}^3}e^{i(\xi_4\phi(x_1,x_2,x_3)+\xi_1x_1+\xi_2x_2+x_3\xi_3)} \eta(x)dx\right|\leq C||\eta||_{C^3(\mathbb{R}^3)}(\log(2+|\xi|))^{\nu}(1+|\xi|)^{-1/h}.
\end{equation}
Moreover, if we restrict ourselves to the direction where $\xi_1=\xi_2=\xi_3=0,$  then the estimate  \eqref{main} is sharp, i.e., for any smooth function $ \eta $ with  $\eta(0)\neq 0$ and a  support in a sufficiently small neighborhood of the origin, the following limit
  \begin{equation}\label{sharp}
  \lim_{\xi_4\to +\infty}\frac{\xi_4^{1/h}}{(\log{\xi_4})^{\nu}}\int_{\mathbb{R}^3}e^{i\xi_4\phi
  (x_1,x_2,x_3)} \eta(x)dx=c
  \end{equation}
  exists, where $c$ is a non-zero constant depending on $\phi$ and $\eta$.
\end{theorem}
A corollary of Theorem \ref{th:osc} and Corollary \ref{esseen} is that the estimate \eqref{main} is stable under sufficiently small analytic perturbation of $\phi$:
\begin{corollary}\label{cor:Karp}
Let  $\phi:\mathbb{R}^3\to \mathbb{R}$ be a  polynomial function satisfying the conditions of Theorem \ref{th:osc}. Let $h=h(\phi)$ be the height of $\phi$ defined in \eqref{height} and $\nu=\nu(\phi)\in \{0,1\}$ be defined as above. Then, for any neighborhood $V \subset \mathbb{C}^3$ of the origin, there exists  $\varepsilon > 0$ and a neighborhood $U \subset \mathbb{R}^3$ of the origin such that for any real-analytic function $\Phi$ which has an analytic continuation to $V$ with $\sup_{z \in V} |\Phi(z)| < \varepsilon$, for any $\eta \in C_0^\infty(U)$, and for $\xi=(\xi_1,\xi_2,\xi_3,\xi_4)\in \mathbb{R}^4,$  the following estimate holds:
\begin{equation}\label{finalb}
\left|\int_{\mathbb{R}^3}e^{i (\xi_4(\phi(x_1, x_2, x_3)+\Phi(x_1, x_2, x_3))+\xi_1x_1+\xi_2x_2+\xi_3x_3} \eta(x)dx\right|\leq C||\eta||_{C^3(\mathbb{R}^3)}(\log(2+|\xi|))^{\nu}(1+|\xi|)^{-1/h}.
\end{equation}
\end{corollary}
Indeed, if the polynomial function $\phi$ satisfies the conditions of Theorem \ref{th:osc}, then, due to Corollary \ref{esseen}, it can be reduced to a polynomial function that depends on two variables  by applying a local analytic  change of variables. We can then use the stability result proved by Karpushkin \cite{Kar} in two dimensions and obtain   the estimate \eqref{finalb}.\\

\begin{proof}[Proof of Theorem \ref{th:osc}]
We denote
$$J(\xi,\phi)\colonequals \int_{\mathbb{R}^3}e^{i(\xi_4\phi(x_1,x_2,x_3)+\xi_1x_1+\xi_2x_2+x_3\xi_3)} \eta(x)dx.$$
Let $|\xi_1|+|\xi_2|+|\xi_3|>\delta |\xi_4|$ for sufficiently small $\delta>0,$ then the estimate \eqref{main} follows by an integration by parts, if $U$ is chosen small enough, see \cite[Chapter VIII]{Stein}.
Therefore, we assume
\begin{equation}\label{xi}
|\xi_1|+|\xi_2|+|\xi_3|\leq \delta |\xi_4|,\,\,\,\,\text{hence}\,\,\,\,|\xi|\thicksim |\xi_4|.
\end{equation}
By Theorem \ref{Bon} there exists an invertible matrix $A\in \mathbb{R}^{3\times 3}$ such that
$ \phi(Ax)$ either depends on at most two variables $x_1$ and $x_2$ or has the form
$$\phi(Ax)=Q_1(x_1)+Q_2(x_1)x_2+Q_3(x_1)x_3,$$
where $Q_1, Q_2, Q_3:\mathbb{R}\to \mathbb{R}$ are polynomial functions.
 In the following  $C$ will be used to denote positive constants, which may vary from one line to the next.

\bigskip
\textbf{Part 1:} Assume $ \phi(Ax)$ depends  on at most two variables $x_1$ and $x_2$. Since a linear change of variables does not change the behavior of $J(\xi,\phi )$ as $|\xi|\to \infty,$ we  prove \eqref{main} and \eqref{sharp} for $ \tilde{\phi}(x_1,x_2)\colonequals \phi(Ax)$ instead of $\phi(x).$   We write
$$
J(\xi,\tilde{\phi})=\int_{\mathbb{R}} e^{i \xi_3x_3}\left(\int_{\mathbb{R}^2}e^{i(\xi_4\tilde{\phi}(x_1,x_2) +\xi_1x_1+\xi_2x_2)} \eta(x)dx_1dx_2\right)
dx_3.
$$
 We use \cite[Theorem 1.1]{Ik.Mu} to estimate the inner integral and obtain
$$|J(\xi,\tilde{\phi})| \leq C (\log(2+|\xi_1|+|\xi_2|+|\xi_4|))^{\nu}(1+|\xi_1|+|\xi_2|+|\xi_4|)^{-1/h}\int_{\mathbb{R}} ||\eta(\cdot, \cdot, x_3)||_{C^3(\mathbb{R}^2)}dx_3,$$
where the constant $C>0$ is independent of $x_3.$   Since   $\eta$ is supported in $U, $  we can estimate
$$\int_{\mathbb{R}} ||\eta(\cdot, \cdot, x_3)||_{C^3(\mathbb{R}^2)}dx_3\leq \mathrm{diam}(U)  ||\eta||_{C^3(\mathbb{R}^2)} ,$$
where $\mathrm{diam}(U)$ denotes the diameter of $U.$ The last two inequalities  and \eqref{xi} yield \eqref{main}.

If $\xi_1=\xi_2=\xi_3=0,$ then
$$
J(\xi_4,\tilde{\phi})=\int_{\mathbb{R}^2}e^{i\xi_4\tilde{\phi}(x_1,x_2) } \tilde{\eta}(x_1,x_2)dx_1dx_2
,
$$
where $\tilde{\eta}(x_1,x_2)\colonequals \int_{\mathbb{R}}\eta(x_1,x_2,x_3)
dx_3.$
Then the existence of  the limit \eqref{sharp}  follows from \cite[Theorem 0.6]{Var}.

\bigskip
\textbf{Part 2:} Let
$ \phi(Ax)$ have the form
$$\phi(Ax)=Q_1(x_1)+Q_2(x_1)x_2+Q_3(x_1)x_3,$$
where $Q_1, Q_2, Q_3:\mathbb{R}\to \mathbb{R}$ are polynomial functions.
Let $\nu_1\geq 2,\nu_2\geq 1,\nu_3\geq 1$  denote the multiplicities of $x_1=0$ as the root of    $Q_1(x_1),$ $Q_2(x_1),$ and $Q_3(x_1),$ respectively.  If  $Q_2$ or  $Q_3$ is identically zero, then $\phi(Ax)$ depends on at most two variables and so Part 1  implies the proof. Hence, we assume that $Q_2(x_1)$ and $Q_3(x_1)$ are not identically zero. $Q_1(x_1)$ can be identically zero. In this case, we formally put $\nu_1=\infty$. Without loss of generality we can assume $\nu_3\geq \nu_2.$  Then  $\min\{\nu_1,\nu_2,\nu_3\}$  can be $\nu_1$ or $\nu_2.$
Hence, we consider two different cases:

\textbf{Case 1:} Let $\nu_1= \min\{\nu_1,\nu_2,\nu_3\}.$
 The proof of Theorem \ref{th:adap} shows that in this case the given coordinate system is adapted
and  $d_x=h(\phi)=\nu_1.$ Since a linear change of variables does not change the behavior of $J(\xi, \phi),$ we prove \eqref{main} and \eqref{sharp} for $\phi(Ax)$ instead of $\phi(x).$ Then the function $\phi(Ax)$ can be written as $\phi(Ax)=x_1^{\nu_1}\tilde Q(x_1, x_2, x_3)$, where $\tilde Q$ is a polynomial function with $\tilde Q(0, 0, 0)\neq0$.   We write
\begin{equation*}
J(\xi, \phi(A\cdot))=\int_{\mathbb{R}^2}e^{i(\xi_2x_2+\xi_3x_3)}\left( \int_{\mathbb{R}}e^{i(\xi_4x_1^{\nu_1}\tilde Q(x_1, x_2, x_3)+\xi_1x_1)} {\eta}(x)dx_1\right)dx_2dx_3.
\end{equation*}
 We apply \cite[Lemma 2.2]{Ik.Mu}  to the inner integral and get
$$|J(\xi, \phi(A\cdot))|\leq  C (1+|\xi_4|)^{-1/\nu_1} \int_{\mathbb{R}^2}||\eta( \cdot,x_2,x_3)||_{C^1(\mathbb{R})}dx_2dx_3  ,$$
where the constant $C>0$ depends only on $\tilde Q(0, 0, 0)$ and $\nu_1.$ We can estimate
$$\int_{\mathbb{R}^2}||\eta(\cdot, x_2,x_3)||_{C^1}dx_2dx_3\leq (\mathrm{diam}(U))^2||\eta||_{C^1(\mathbb{R}^3)}.$$
This estimate and \eqref{xi} yields \eqref{main}.

If $\xi_1=\xi_2=\xi_3=0,$ then we write
\begin{equation*}
\lim_{\xi_4\to +\infty} \xi_4^{1/\nu_1} J(\xi_4, \phi(A\cdot))=\lim_{\xi_4\to +\infty} \int_{\mathbb{R}^2}\left( \xi_4^{1/\nu_1} \int_{\mathbb{R}}e^{i\xi_4x_1^{\nu_1}\tilde Q(x_1, x_2, x_3)} {\eta}(x)dx_1\right)dx_2dx_3.
\end{equation*}
 Erd\'elyi's lemma (see \cite{Fed77, Ik.Mu}) implies that there is a non-zero constant $C$ (independent of $x_2$ and $x_3$) such that
 $$\lim_{\xi_4\to +\infty} \xi_4^{1/\nu_1} \int_{\mathbb{R}}e^{i\xi_4x_1^{\nu_1}\tilde Q(x_1, x_2, x_3)} {\eta}(x_1,x_2,x_3)dx_1=C\eta(0, x_2, x_3).$$
   Thus, if  the support of $\eta$ is sufficiently small and $\eta(0)\neq 0,$ by integrating the above equation with respect to $x_2$ and $x_3$  we obtain \eqref{sharp}.

\textbf{Case 2:} Let $\nu_2= \min\{\nu_1,\nu_2,\nu_3\}.$   
If $\nu_1=\nu_2,$ then Case 1 implies the proof. Hence, we assume $\nu_1>\nu_2.$ Then the proof of Theorem \ref{th:adap} shows that the expression of $\phi$ in an adapted coordinate system is
\begin{align*}
\phi(Bx)=&x_1^{\nu_2}x_2+x_1^{\nu_1}(\tilde Q_1(x_1)-c_1)-\frac{c_1}{c_2}x_1^{\nu_2+1}(\tilde Q_2(x_1)-c_2) \\& +\frac{x_1^{\nu_2}x_2}{c_2}(\tilde Q_2(x_1)-c_2)+\left(x_1^{\nu_3}(\tilde Q_3(x_1)-c_3)-\frac{c_3x_1^{\nu_2} }{c_2}(\tilde Q_2(x_1)-c_2)\right)x_3,
\end{align*}
where $B\in \mathbb{R}^{3\times 3}$ is an invertible matrix. The distance is $d_x=h(\phi)=\nu_2.$ Since a linear change of variables does not change the asymptotic behavior of $J(\xi, \phi), $ we prove \eqref{main} and \eqref{sharp} for $\phi(Bx)$ instead of $\phi(x).$
We write
 \begin{equation*}
 J(\xi, \phi(B\cdot))=\int_{\mathbb{R}} e^{i\xi_3x_3} \left( \int_{\mathbb{R}^2}e^{i(\xi_4\phi(Bx)+\xi_1x_1+\xi_2x_2)} {\eta}(x)dx_1dx_2\right)dx_3.
 \end{equation*}
 If we consider $x_3$ as a parameter, Proposition \ref{pr:a} implies that  the given coordinate system with respect to   $x_1$ and $x_2$ is adapted for $\phi(Bx)$ and   the principal face of the Newton polyhedron is unbounded if $\nu_2>1$ and a vertex if $\nu_2=1.$
 Note that $\phi(Bx)$ and its gradient with respect to $x_1$ and $x_2$ are zero for all $x=(0,0,x_3)\in \mathbb{R}^3.$
 This let us apply \cite[Theorem 1.1]{Ik.Mu} to the inner integral in the above equation
 $$J(\xi, \phi(B\cdot))\leq C (1+|\xi_4|)^{-1/h}\int_{\mathbb{R}}||\eta(\cdot,\cdot, x_3)||_{C^3(\mathbb{R}^2)}dx_3,$$
 where the constant $C>0$ is independent of $x_3.$
 We can estimate
 $$\int_{\mathbb{R}}||\eta(\cdot, \cdot, x_3)||_{C^3(\mathbb{R}^2)}dx_3\leq \mathrm{diam}(U)||\eta||_{C^3(\mathbb{R}^3)}.$$
 These estimates and \eqref{xi} yield \eqref{main}.

 If $\xi_1=\xi_2=\xi_3=0,$ then we write
 \begin{eqnarray*}
 \lim_{\xi_4\to +\infty} \xi_4^{1/\nu_2} J(\xi_4, \phi(B\cdot))=\lim_{\xi_4\to +\infty} \int_{\mathbb{R}}  \left(\xi_4^{1/\nu_2} \int_{\mathbb{R}^2}e^{i\xi_4\phi(Bx)} {\eta}(x)dx_1dx_2\right)dx_3.
 \end{eqnarray*}
 According to  \cite[Theorem 0.6]{Var}, for every fixed $x_3,$ the limit
\begin{eqnarray*}
 \lim_{\xi_4\to +\infty} \xi_4^{1/\nu_2} \int_{\mathbb{R}^2}e^{i\xi_4\phi(Bx)} {\eta}(x)dx_1dx_2
  \end{eqnarray*}
 exists and is non-zero  if the support of $\eta$ is sufficiently small and $\eta(0)\neq 0.$ Then integrating with respect to $x_3$ and using the Lebesgue dominated convergence theorem, we obtain \eqref{sharp}.


\end{proof}





\section{Estimates for maximal operators }
In this section, we study the $L^p$-boundedness of maximal operators associated with hypersurfaces $S$ in $\mathbb{R}^4$.   Let $d\sigma$ denote the surface measure on $S$ and let
 $\rho\in C_0^{\infty}(S)$ be non-negative. We recall the averaging operator
 \begin{equation*}
 \mathcal{A}_tf(y)\colonequals \int_Sf(y-tx)\rho(x)d\sigma(x),\,\,\,\,t>0,\,\,\,\,f\in C^{\infty}_0(\mathbb{R}^4)
 \end{equation*}
 and  the corresponding  maximal operator
 \begin{equation*}
 \mathcal{M}f(y)\colonequals \sup_{t>0}|\mathcal{A}_tf(y)|, \,\,\,\,f\in C^{\infty}_0(\mathbb{R}^4).
 \end{equation*}
As mentioned in the introduction, using a partition of unity argument, we can assume that $\rho$ is supported in a sufficiently small neighborhood of a fixed point $x^0\in S.$ The transversality assumption let us find a linear change of coordinates in $\mathbb{R}^4$ so that in the new coordinates $S$ can be locally represented as a graph
\begin{equation}\label{S4}
S=\{(x_1,x_2,x_3, 1+\phi(x_1,x_2,x_3))\in \mathbb{R}^{4}:\,(x_1,x_2,x_3)\in U\},
\end{equation}
where $U\subset \mathbb{R}^3$ is a sufficiently small neighborhood of the origin.
 We assume $\phi$ is a polynomial function and satisfies \eqref{phi=0} and \eqref{det=0}.
  Then the maximal  operator  can be written as
 \begin{equation}\label{M}
 \mathcal{M}f(y)= \sup_{t>0}\left| \int_{\mathbb{R}^3}f(y_1-tx_1, y_2-tx_2, y_3-tx_3, y_4-t(1+\phi(x)))\eta(x)dx\right|, \,\,f\in C^{\infty}_0(\mathbb{R}^4),
 \end{equation}
  where  $$\eta(x_1,x_2,x_3)\colonequals \rho(x_1,x_2,x_3,1+\phi(x_1,x_2,x_3))\sqrt{1+|\nabla \phi(x_1,x_2,x_3)|^2}\in C^{\infty}_0(U).$$

We first state the main results of this section:

  \begin{theorem}\label{th:Max}
Let $\phi:\mathbb{R}^3\to \mathbb{R}$ be a polynomial function such that $\phi(0)=0$ and $\nabla\phi(0)=0.$ Assume
$$\det(D^2\phi(x))=0$$
holds for all $x\in \mathbb{R}^3.$ Let $h(\phi)$ be the height of $\phi$ defined in \eqref{height}. Then there exists a neighborhood $U\subset \mathbb{R}^3$ of the origin such that, for every non-negative $\eta\in C^{\infty}_0(U),$ the associated maximal operator $\mathcal{M}$ in \eqref{M} is bounded
on  $L^p(\mathbb{R}^4)$ for $p> \max\{h(\phi),2\}.$   

  \end{theorem}
Iosevich and Sawyer \cite{Io.Saw} proved that if the maximal operator $\mathcal{M}$ with $\eta(0) > 0$ is bounded on $L^p(\mathbb{R}^{n+1})$, then
  \begin{equation}\label{d(x,H)}
  d(x,H)^{-1/p}\in L^{1}_{\mathrm{loc}}(S),
    \end{equation}
  where $H$ is any hyperplane not passing through the origin, and $d(x,H)$ denotes the distance from $x\in S$ to $H$.   As shown in  the proof of Theorem 2 in  \cite{Io.Saw}, for a hypersurface $S$ given in \eqref{S4}, the condition    \eqref{d(x,H)} is equivalent to
  \begin{equation}\label{-1/p}
  |\phi|^{-\frac{1}{p}}\in L^{1}(U).
   \end{equation}
  We check for which $p\geq 1$ the condition \eqref{-1/p}  holds true.
  \begin{proposition}\label{prop:max}
  Let $\phi:\mathbb{R}^3\to \mathbb{R}$ be a polynomial such that $\phi(0)=0$ and $\nabla\phi(0)=0.$ Assume
$$\det(D^2\phi(x))=0$$
holds for all $x\in \mathbb{R}^3.$ Let $h(\phi)$ be the height of $\phi$ defined in \eqref{height}.
\begin{itemize}
\item[(i)]  For every $p> h(\phi)$ and sufficiently small neighborhood $U\subset \mathbb{R}^3$ of the origin, we have
\begin{equation*}
\int_U |\phi(x)|^{-\frac{1}{p}}dx<\infty.
\end{equation*}
\item[(ii)]  For every $p\leq h(\phi)$ and sufficiently small neighborhood $U\subset \mathbb{R}^3$ of the origin, we have
\begin{equation*}
\int_U |\phi(x)|^{-\frac{1}{p}}dx=\infty.
\end{equation*}
\end{itemize}

  \end{proposition}
  As an immediate consequence of Theorem \ref{th:Max}, \cite[Theorem 2]{Io.Saw}, and Proposition \ref{prop:max} $(ii)$ we obtain the following corollary:
  \begin{corollary}\label{cor:max}
  Let $\phi:\mathbb{R}^3\to \mathbb{R}$ be a polynomial such that $\phi(0)=0$ and $\nabla\phi(0)=0.$ Assume
$$\det(D^2\phi(x))=0$$
holds for all $x\in \mathbb{R}^3.$ Let $h(\phi)$ be the height of $\phi$ defined in \eqref{height}. Assume there exists a neighborhood $U\subset \mathbb{R}^3$ of the origin and  $\eta\in C^{\infty}_0(U)$ with $\eta(0)>0$ such that the associated maximal operator $\mathcal{M}$ in \eqref{M} is bounded
on  $L^p(\mathbb{R}^4)$ for some  $p> 1. $ Then we necessarily have  $p>h(\phi).$   In particular, if $h(\phi)\geq 2,$ then the maximal operator $\mathcal{M}$ is bounded if and only if $p>h(\phi).$
  \end{corollary}
  \begin{Remark}
   $(i)$ Iosevich and Sawyer \cite{Io.Saw} conjectured that for any smooth hypersurface $S$ and $p>2,$ the condition \eqref{d(x,H)} is necessary and sufficient for the boundedness of the maximal operator $\mathcal{M}$  on $L^p(\mathbb{R}^{n+1}). $ Corollary \ref{cor:max} confirms this conjecture in our setting. 

   $(ii)$ As we mentioned in the introduction,  Stein, Iosevich, and Sawyer  conjectured   that if \eqref{O} holds for some $0<q\leq 1/2,$ then  $\mathcal{M}$ is bounded for $p>1/q.$ Theorem \ref{th:osc} shows that \eqref{O} holds for all $q<1/h(\phi).$ By Theorem \ref{th:Max}, if $h(\phi)\geq 2,$ then   $\mathcal{M}$ is bounded for all $p>1/q>h(\phi).$ This confirms the Stein-Iosevich-Sawyer  conjecture in our setting.
  \end{Remark}

  Corollary \ref{cor:max} provides complete answer to the $L^p$ boundedness of the maximal operator $\mathcal{M}$ when $h(\phi)\geq 2.$ Next we consider the case $h(\phi)< 2.$ As observed in \cite{P1, P2} for hypersurfaces in $\mathbb{R}^3$, in this case the height is not the main controlling quantity and  the number of non-vanishing principal curvatures of $S$ also influences.  It is obvious that the condition $$\det(D^2\phi(x))=0, \,\,\,\,\,\, \forall\,x\in \mathbb{R}^3$$ means that one of the principal curvatures of  $S$ in \eqref{S4} is always zero. Hence, only two principal curvatures can be non-zero. If these two principal curvatures of $S$ are non-zero at $x^0=(0,0,0,1)\in S$, i.e., the rank of $D^2\phi (0)$ is two, then \cite[Theorem 2]{Gr} provides that  $\mathcal{M}$ is bounded on $L^p(\mathbb{R}^4)$ if  $p>\frac{3}{2}.$  We study the case  $h(\phi)< 2$ by assuming that  all principal curvatures are zero at $x^0=(0,0,0,1)\in S,$ i.e., $D^2\phi (0)=0.$ The case where only one principal curvature of $S$ vanishes and the other does not turns out to be difficult to analyze and so this case remains open, even for hypersurfaces in $\mathbb{R}^3$, see \cite{P2}.
  \begin{theorem}\label{th:h<2}
  Let $\phi:\mathbb{R}^3\to \mathbb{R}$ be a polynomial such that $\phi(0)=0$ and $\nabla\phi(0)=0.$ Assume
$$\det(D^2\phi(x))=0$$
holds for all $x\in \mathbb{R}^3.$ Let the height $h(\phi) $ of $\phi$ is strictly less than 2.
\begin{itemize}
\item[(i)] Assume the rank of $D^2\phi(0)$ is two. Then there exists a neighborhood $U\subset \mathbb{R}^3$ of the origin such that, for every non-negative $\eta\in C^{\infty}_0(U),$ the associated maximal operator $\mathcal{M}$ in \eqref{M} is bounded
on  $L^p(\mathbb{R}^4)$  if  $p>3/2.$
\item[(ii)] Assume that  $D^2\phi(0)=0.$  Then there exists a neighborhood $U\subset \mathbb{R}^3$ of the origin such that, for every non-negative $\eta\in C^{\infty}_0(U)$ with $\eta(0)>0$, the associated maximal operator $\mathcal{M}$  in \eqref{M} is bounded
on  $L^p(\mathbb{R}^4)$ if and only if  $p> h(\phi).$
\end{itemize}
  \end{theorem}

\bigskip

  We now start proving the above results:
 \begin{proof}[\textbf{Proof of Theorem \ref{th:Max}}]
 By Theorem \ref{Bon} there exists an invertible matrix $A\in \mathbb{R}^{3\times 3}$ such that
$ \phi(Ax)$ either depends on at most two variables $x_1$ and $x_2$ or has the form in \eqref{form}. Hence, we divide the proof into two parts.

\bigskip
\textbf{Part 1:} Assume $ \phi(Ax)$ depends on at most two variables $x_1$ and $x_2$. Since the  linear change of variables  $Ax $ does not affect the $L^p$-boundedness of $\mathcal{M}$,
 we can assume $\phi$ depends on  at most two variables $x_1$ and $x_2$.
The averaging operator can be written as
$$\mathcal{A}_tf(y)=
\int_{\mathbb{R}^3}f(y_1-tx_1,y_2-tx_2,y_3-tx_3, y_4-t(1+ \phi(x_1,x_2))) \eta(x)dx.$$
 We change the variables
 \begin{align*}
 x_1&=x_1,\\
 x_2&=x_2,\\
 x_3&=(1+ \phi(x_1,x_2))\tan{\theta},
\end{align*}
where $\theta$ varies in a sufficiently small interval $(a,b)\subset (-\pi/2,\pi/2)$ which contains $0.$  We require that $U$ is  sufficiently small so that this change of variables is well defined. Then we can write
 \begin{align*}
 \mathcal{A}_t f(y)=\int_{a}^{b}
\int_{\mathbb{R}^2}f(y_1-tx_1,y_2-tx_2,y_3-t(1+ \phi(x_1,x_2))&\tan{\theta}, y_4-t(1+ \phi(x_1,x_2)))\\
&\,\,\,\,\,\times \tilde{\eta} (x_1,x_2, \theta)dx_1dx_2d\theta,
\end{align*}
where $$\tilde{\eta} (x_1,x_2, \theta)\colonequals \frac{1}{\cos^{2}{\theta}} (1+ \phi(x_1,x_2)) \eta (x_1,x_2, (1+ \phi(x_1,x_2))\tan{\theta}). $$
We denote
$$\mathcal{A}_t^{\theta} f(y)\colonequals
\int_{\mathbb{R}^2}f(y_1-tx_1,y_2-tx_2,y_3-t(1+ \phi(x_1,x_2))\tan{\theta}, y_4-t(1+ \phi(x_1,x_2))) \tilde{\eta} (x_1,x_2, \theta)dx_1dx_2.$$
For every fixed $\theta\in (a,b),$ we  define the rotation operator
$$R^{\theta}f(y)\colonequals f(y_1,y_2, y_3\cos{\theta}-y_4\sin{\theta}, y_3\sin{\theta}+y_4\cos{\theta}).$$
One can easily check that $R^{\theta}$ and its inverse $R^{-\theta}$ are isometric in $L^p(\mathbb{R}^{4})$ for all $p\geq 1.$
This operator let us  write
$$R^{-\theta}\mathcal{A}_t^{\theta}R^{\theta}f(y)=
\int_{\mathbb{R}^2}f(y_1-tx_1,y_2-tx_2,y_3, y_4-\frac{t}{\cos{\theta}}(1+ \phi(x_1,x_2))) \tilde{\eta} (x_1,x_2, \theta)dx_1dx_2.$$
Observe that the last operator acts only on the first, second, and fourth  variables of $f.$
For every fixed $\theta\in (a,b)$ and  $p\geq 1,$ we define the scaling operator
$$T_{\theta}f(y)\colonequals (\cos{\theta})^{\frac{1}{p}}f(y_1,y_2,y_3, y_4\cos{\theta})$$
and its inverse
$$T_{\theta}^{-1}f(y)\colonequals \frac{1}{(\cos{\theta})^{\frac{1}{p}}}f\big(y_1,y_2,y_3, \frac{y_4}{\cos{\theta}}\big).$$
$T_{\theta}$ and  $T_{\theta}^{-1}$ are isometric in $L^p(\mathbb{R}^{4})$ for a fixed $\theta\in (a, b)\subset (-\pi/2, \pi/2)$. We have
$$ T_{\theta}^{-1}R^{-\theta}\mathcal{A}_t^{\theta}R^{\theta}T_{\theta} f(y)=
\int_{\mathbb{R}^2}f(y_1-tx_1,y_2-tx_2,y_3, y_4-t(1+ \phi(x_1,x_2))) \tilde{\eta} (x_1,x_2, \theta)dx_1dx_2.$$
We denote $A^{\theta}_tf(y)\colonequals T_{\theta}^{-1}R^{-\theta}\mathcal{A}_t^{\theta}R^{\theta}T_{\theta} f(y).$
 By considering $y_3$ and $\theta$ as  parameters, we can apply \cite[Theorem 1.3]{I.K.M} if $U$ is sufficiently small, and obtain that, for $p>\max\{h(\phi),2\},$
\begin{equation*}\int_{\mathbb{R}^3}\big(\sup_{t>0}|A^{\theta}_tf(y)| \big)^pdy_1dy_2dy_4\leq C_p \int_{\mathbb{R}^3}|f(y_1,y_2,y_3,y_4)|^pdy_1dy_2dy_4,
\end{equation*}
where $C_p$ is independent of  $\theta$ and $y_3.$
 Integrating this equation with respect to  $y_3,$ we obtain that the operator $\sup_{t>0}|{A}_t^{\theta}f|$ is bounded on $L^p(\mathbb{R}^4)$ for $p> \max\{h(\phi),2\}.$
With the above notations we can write
$$\mathcal{A}_t f(y)=\int_a^b T_{\theta}R^{\theta}{A}_t^{\theta}R^{-\theta}T_{\theta}^{-1} f(y)d\theta.  $$
Since $T_{\theta}$ and $R^{\theta}$ do not depend on $t,$ we can estimate
$$\sup_{t>0}|\mathcal{A}_t f(y)|\leq \int_a^b\sup_{t>0}| T_{\theta}R^{\theta}{A}_t^{\theta}R^{-\theta}T_{\theta}^{-1} f(y)|d\theta\leq \int_a^b T_{\theta}R^{\theta}\sup_{t>0}| {A}_t^{\theta}R^{-\theta}T_{\theta}^{-1} f(y)|d\theta .  $$
 For $p>\max\{h(\phi),2\},$ we estimate the $L^p$ norm of $\sup_{t>0}|\mathcal{A}_t f(y)|$
  \begin{align*}\int_{\mathbb{R}^4}\big(\sup_{t>0}|\mathcal{A}_tf(y)| \big)^pdy & \leq (b-a)^{p-1}
 \int_a^b\left[ \int_{\mathbb{R}^4}
 \left(T_{\theta}R^{\theta}
 \sup_{t>0}| {A}_t^{\theta}R^{-
 \theta}T_{\theta}^{-1} f(y)|
 \right)^pdy\right]d\theta\\
 &=(b-a)^{p-1}
 \int_a^b\left[ \int_{\mathbb{R}^4}
 \left(
 \sup_{t>0}| {A}_t^{\theta}R^{-
 \theta}T_{\theta}^{-1} f(y)|
 \right)^pdy\right]d\theta, \end{align*}
where we used  the H\"older inequality and the fact that $T_{\theta}R^{\theta}$ is isometric  in $L^p(\mathbb{R}^4).$ Since $\sup_{t>0}|{A}_t^{\theta}f|$ is bounded on $L^p(\mathbb{R}^4)$ for $p> \max\{h(\phi), 2\}$ and $R^{-
 \theta}T_{\theta}^{-1}$ is isometric in $L^p(\mathbb{R}^4),$ we get
 \begin{align*}\int_{\mathbb{R}^4}\big(\sup_{t>0}|\mathcal{A}_tf(y)| \big)^pdy & \leq C_p(b-a)^{p-1}
\int_a^b \int_{\mathbb{R}^4}
 |R^{-
 \theta}T_{\theta}^{-1} f(y)|^pdyd\theta\\
 &=C_p(b-a)^{p}
  \int_{\mathbb{R}^4}
| f(y)|^pdy.
 \end{align*}
 Hence, $\mathcal{M}$ is bounded on $L^p(\mathbb{R}^4)$ for $p> \max\{h(\phi),2\}.$

\bigskip
\textbf{Part 2:} Let
$ \phi(Ax)$ have the form in \eqref{form}
$$\phi(Ax)=Q_1(x_1)+Q_2(x_1)x_2+Q_3(x_1)x_3,$$
where $Q_1, Q_2, Q_3:\mathbb{R}\to \mathbb{R}$ are polynomial functions.
 Let $\nu_1,\nu_2,\nu_3$  denote the multiplicities of $x_1=0$ as the root of $Q_1(x_1),$ $Q_2(x_1),$ and $Q_3(x_1),$ respectively. If  $Q_2(x_1)$ or  $Q_3(x_1)$ is identically zero, then $\phi(Ax)$ depends on at most two variables and so Part 1  implies the proof. Hence, we assume that $Q_2(x_1)$ and $Q_3(x_1)$ are not identically zero. $Q_1(x_1)$ can be identically zero. In this case, we formally put $\nu_1=\infty$.
  The conditions $\phi(0)$ and $\nabla \phi(0)=0$ imply $\nu_1\geq 2, \nu_2\geq 1,$ and $\nu_3\geq 1.$  We can write
\begin{eqnarray*}
Q_i(x_1)=x_1^{\nu_i}\tilde Q_i(x_1),\,\,\,\,i\in \{1,2,3\},
\end{eqnarray*}
where $\tilde Q_i$ are  polynomial functions. According to our assumption we have $\tilde Q_i(0)\neq 0$    for $i\in \{2,3\},$ and $\tilde Q_1(0)\neq 0$  if  $Q_1$ is not identically zero.
Without loss of generality we can assume $\nu_3\geq \nu_2.$  Then  $\min\{\nu_1,\nu_2,\nu_3\}$  can be $\nu_1$ or $\nu_2.$
Hence, we consider two  cases:

\textbf{Case 1:} Let $\nu_1= \min\{\nu_1,\nu_2,\nu_3\}.$
 The proof of Theorem \ref{th:adap} shows that in this case the given coordinate system is adapted
and  $d_x=h(\phi)=\nu_1\geq 2.$
Since the  linear change of variables  $Ax $ does not affect the $L^p$-boundedness of $\mathcal{M}$,  we can assume
\begin{equation*}\label{firstcase}
\phi(x)=Q_1(x_1)+Q_2(x_1)x_2+Q_3(x_1)x_3=x_1^{\nu_1}\tilde Q(x_1, x_2, x_3),
\end{equation*}
where $\tilde Q$ is a polynomial function with $\tilde Q(0, 0, 0)\neq0$.


The averaging operator can be written as
\begin{eqnarray*}
\mathcal{A}_tf(y)=
\int_{\mathbb{R}^3}f(y_1-tx_1,y_2-tx_2,y_3-tx_3, y_4-t(1+ \phi(x_1,x_2,x_3))) \eta(x)dx.
\end{eqnarray*}
 One can check that the hyperplane $\{(t_1, t_2, t_3)\in \mathbb{R}^3_{+}:\, t_1=\nu_1\}$  touches  the Newton polyhedron of $\phi$ only at $(\nu_1,0,0).$
We choose a smooth non-negative function $\chi_0$  such that
\begin{eqnarray*}
\chi_0(x_1)=\begin{cases}1,\,\,\,\,\text{for}\,\,\,\,|x_1|\leq 1,\\
0,\,\,\,\,\text{for}\,\,\,\,|x_1|\ge 2.\end{cases}
\end{eqnarray*}
We denote
\begin{eqnarray*}
\chi_1(x_1)\colonequals \chi_0(x_1)-\chi_0(2x_1).
\end{eqnarray*}
Then $\chi_1(x_1)$ is supported in
$$D\colonequals \{2^{-1}\leq  |x_1|\leq 2\}.$$
One can check that
\begin{eqnarray*}
\sum_{k=k_0}^{\infty}\chi_1(2^kx_1)=1\,\,\,\,\,\,\text{for}\,\,\,\,0<|x_1|\leq 2^{-k_0}.
\end{eqnarray*}
We assume that $U$ is sufficiently small so that, for every $x\in U,$ we have $|x_1|\leq 2^{-k_0}.$ Notice that by choosing $U$ small, we can choose $k_0\in \mathbb{N}$ as large as we need. We can then decompose $\mathcal{A}_t$  as
\begin{eqnarray*}
\mathcal{A}_tf(y)=\sum_{k=k_0}^{\infty} \mathcal{A}_t^k f(y),
\end{eqnarray*}
where
\begin{eqnarray*}
 \mathcal{A}_t^k f(y)\colonequals \int_{\mathbb{R}^3}f(y_1-tx_1,y_2-tx_2,y_3-tx_3, y_4-t(1+ \phi(x_1,x_2,x_3))) \eta(x)\chi_1(2^kx_1)dx.
 \end{eqnarray*}
 We change the variables as $(x_1, x_2, x_3)\to (2^{-k}x_1, x_2, x_3),$ then
\begin{eqnarray*}
\mathcal{A}_t^kf(y)=
2^{-k}\int_{\mathbb{R}^3}f(y_1-2^{-k}tx_1, y_2-tx_2, y_3-tx_3, y_4-t(1+ 2^{-\nu_1k}\phi_k(x))) \eta_k(x)\chi_1(x_1)dx,
\end{eqnarray*}
where $\phi_k(x_1, x_2, x_3)\colonequals  x_1^{\nu_1} \tilde Q(2^{-k}x_1, x_2, x_3)$ and $\eta_k(x_1, x_2, x_3)\colonequals \eta(2^{-k}x_1, x_2, x_3). $ Assume $p>\nu_1.$ We define the scaling operator
\begin{eqnarray*}
T^kf(y_1,y_2,y_3,y_4)\colonequals 2^{\frac{k}{ p}} f(2^ky_1,  y_2,  y_3, y_4).
\end{eqnarray*}
Note that $T^k$ acts isometrically on $L^p(\mathbb{R}^4),$ and
\begin{eqnarray*}
T^{-k} \mathcal{A}_t^k T^kf(y)=2^{-k}\int_{\mathbb{R}^3}f(y_1-tx_1,y_2-tx_2,y_3-tx_3, y_4-t(1+ 2^{-k\nu_1}\phi_k(x))) \eta_k(x)\chi_1(x_1)dx.
\end{eqnarray*}
If $U$ is sufficiently small, we have
$\partial_{x_1}^{2} \phi_k(x) \neq 0$
for all $x \in D \times U_1$, where $U_1$ is the image of $U$ under the projection $\mathbb{R}^3\mapsto \mathbb{R}^2$ given by $(x_1, x_2, x_3)\to (x_2, x_3)$.
 If $k_0$ is sufficiently large,  then \cite[Theorem 7.1]{I.U} (which is an extension of \cite{Sogge} and \cite[Proposition 4.5]{I.K.M}) implies that, for $p>2,$
\begin{eqnarray*}
\big|\big|\sup_{t>0}|T^{-k} \mathcal{A}_t^k T^kf|\big|\big|_{L^p}\leq C 2^{k\left(\frac{\nu_1}{ p}-1\right)}||f||_{L^p},\,\,\,\,\forall\,k\geq k_0,
\end{eqnarray*}
where $C>0$ is independent of $k$. Since $T^k$ acts isometrically on $L^p(\mathbb{R}^4),$ this is equivalent to
\begin{eqnarray*}
\big|\big|\sup_{t>0}| \mathcal{A}_t^k f|\big|\big|_{L^p}\leq C 2^{k\left(\frac{\nu_1}{ p}-1\right)}||f||_{L^p}.
\end{eqnarray*}
If $p>\nu_1\geq 2,$   we can sum over all $k\geq k_0$ and obtain the desired estimate for $\mathcal{M}.$

\textbf{Case 2:}  Let $\nu_2= \min\{\nu_1,\nu_2,\nu_3\}.$ If $\nu_1=\nu_2, $ then Case 1 provides the proof. Therefore, we assume $\nu_1>\nu_2.$

 Note that, if $\nu_2=1$ and $U$ is sufficiently small,  then the rank of $D^2 \phi(x)$ is two for all $x\in U.$ Hence,   the hypersurface has two non-vanishing principal curvatures. Consequently, due to the classical result by Greenleaf in \cite{Gr} the maximal operator is bounded for $p>\frac{3}{2}$.

We now consider the case  $\nu_2\ge2$.
The proof of Theorem \ref{th:adap} shows that the expression of $\phi$ in an adapted coordinate system is
\begin{align*}
\phi(Bx)=&x_1^{\nu_2}x_2+x_1^{\nu_1}(\tilde Q_1(x_1)-c_1)-\frac{c_1}{c_2}x_1^{\nu_2+1}(\tilde Q_2(x_1)-c_2) \\& +\frac{x_1^{\nu_2}x_2}{c_2}(\tilde Q_2(x_1)-c_2)+\left(x_1^{\nu_3}(\tilde Q_3(x_1)-c_3)-\frac{c_3x_1^{\nu_2} }{c_2}(\tilde Q_2(x_1)-c_2)\right)x_3,
\end{align*}
where $B\in \mathbb{R}^{3\times 3}$ is an invertible matrix.  The polynomials $\tilde{Q}_i$ ($i\in \{1,2,3\}$) and the constants $c_i$ are defined in the proof of Theorem \ref{th:adap}.
Since the linear change of variables $Bx$ does not affect the $L^p$-boundedness of $\mathcal{M}$, we can assume
\begin{align*}
\phi(x)=&x_1^{\nu_2}x_2+x_1^{\nu_1}(\tilde Q_1(x_1)-c_1)-\frac{c_1}{c_2}x_1^{\nu_2+1}(\tilde Q_2(x_1)-c_2) \\& +\frac{x_1^{\nu_2}x_2}{c_2}(\tilde Q_2(x_1)-c_2)+\left(x_1^{\nu_3}(\tilde Q_3(x_1)-c_3)-\frac{c_3x_1^{\nu_2} }{c_2}(\tilde Q_2(x_1)-c_2)\right)x_3.
\end{align*}
Note that every monomial  of
\begin{align*}
\phi_r(x)\colonequals & x_1^{\nu_1}(\tilde Q_1(x_1)-c_1)-\frac{c_1}{c_2}x_1^{\nu_2+1}(\tilde Q_2(x_1)-c_2) \\& +\frac{x_1^{\nu_2}x_2}{c_2}(\tilde Q_2(x_1)-c_2)+\left(x_1^{\nu_3}(\tilde Q_3(x_1)-c_3)-\frac{c_3x_1^{\nu_2} }{c_2}(\tilde Q_2(x_1)-c_2)\right)x_3
\end{align*}
has a degree at least $\nu_2+2.$ The distance is   $d_x=h(\phi)=\nu_2\ge 2.$
One can check that the hyperplane $\{(t_1, t_2, t_3)\in \mathbb{R}^3_{+}:\,\frac{t_1}{\nu_2+1/2}+\frac{t_2}{2\nu_2+1}=1\}$  touches  the Newton polyhedron of $\phi$ only at $(\nu_2,1,0).$ 
Similar to Case 1, we choose a smooth non-negative function $\chi_1$ which is
 supported in the annulus
\begin{eqnarray*}
D\colonequals \big\{(x_1,x_2)\in \mathbb{R}^2:\,\,\frac{1}{4}\leq x_1^2+x_2^2\leq 4\big\}
\end{eqnarray*}
and satisfies
\begin{eqnarray*}
\sum_{k=k_0}^{\infty} \chi_1\left(2^{\frac{k}{\nu_2+1/2}}x_1, 2^{\frac{k}{2\nu_2+1}}x_2\right)=1\,\,\,\,\text{for all}\,\,\,\, (0, 0)\neq (x_1, x_2)\quad \mbox{with} \quad (x_1, x_2, x_3) \in U.
\end{eqnarray*}
 Notice that by choosing $U$ small, we can choose $k_0\in \mathbb{N}$ as large as we need. We can then decompose $\mathcal{A}_t$  as
$$\mathcal{A}_tf(y)=\sum_{k=k_0}^{\infty} \mathcal{A}_t^k f(y), $$
where
\begin{align*}
 \mathcal{A}_t^k f(y)\colonequals \int_{\mathbb{R}^3}f(y_1-tx_1,y_2-tx_2,y_3-tx_3 &, y_4-t(1+ \phi(x_1,x_2,x_3))) \\
 &\times \eta(x)\chi_1\left(2^{\frac{k}{\nu_2+1/2}}x_1, 2^{\frac{k}{2\nu_2+1}}x_2\right)dx.
\end{align*}
We apply the  change of variables $(2^{\frac{k}{\nu_2+1/2}}x_1, 2^{\frac{k}{2\nu_2+1}}x_2, x_3)\to (x_1, x_2, x_3) ,$ then
\begin{align*}
 \mathcal{A}_t^k f(y)= 2^{-\frac{3k}{2\nu_2+1}}\int_{\mathbb{R}^3}f(y_1-2^{-\frac{k}{\nu_2+1/2}}tx_1, y_2-2^{-\frac{k}{2\nu_2+1}}tx_2, y_3- tx_3&, y_4-t(1+ 2^{-k}\phi_k(x)))\\
 &\times  \eta_k(x)\chi_1(x_1, x_2)dx,
 \end{align*}
where
\begin{eqnarray*}
\phi_k(x)\colonequals 2^k\phi(2^{-\frac{k}{\nu_2+1/2}}x_1, 2^{-\frac{k}{2\nu_2+1}}x_2,  x_3)
 \end{eqnarray*}
 and
\begin{eqnarray*}
 \eta_k(x)\colonequals \eta(2^{-\frac{k}{\nu_2+1/2}}x_1, 2^{-\frac{k}{2\nu_2+1}}x_2, x_3).
\end{eqnarray*}
Assume $1\le p<\infty.$ We define the scaling operator
\begin{eqnarray*}
T^kf(y_1, y_2, y_3, y_4)\colonequals 2^{\frac{3k}{p(2\nu_2+1)}} f(2^{\frac{k}{\nu_2+1}}y_1, 2^{\frac{k}{\nu_2+1}}y_2, y_3, y_4).
\end{eqnarray*}
Note that $T^k$ acts isometrically on $L^p(\mathbb{R}^4),$ and
\begin{eqnarray*}
T^{-k} \mathcal{A}_t^k T^kf(y)=2^{-\frac{3k}{2\nu_2+1}}\int_{\mathbb{R}^3}f(y_1-tx_1,y_2-tx_2,y_3-tx_3, y_4-t(1+ 2^{-k}\phi_k(x))) \eta_k(x)\chi_1(x)dx.
\end{eqnarray*}
We have
$$\phi_k(x)=x_1^{\nu_2}x_2+2^k\phi_r (2^{-\frac{k}{\nu_2+1/2}}x_1, 2^{-\frac{k}{2\nu_2+1}}x_2, x_3).$$
One can check that $\phi_r(x) $ is divisible by $x_1^{\nu_2+1}$ and so we can write  $$\phi_r(x)=x_1^{\nu_2+1}\tilde{\phi}_r(x),$$
where $\tilde{\phi}_r(x)$ is a polynomial function. This shows
\begin{equation}\label{phrk}
2^k\phi_r (2^{-\frac{k}{\nu_2+1/2}}x_1, 2^{-\frac{k}{2\nu_2+1}}x_2, x_3)=2^{-\frac{k}{2\nu_2+1}}x_1^{\nu_2+1}\tilde{\phi}_r (2^{-\frac{k}{\nu_2+1/2}}x_1, 2^{-\frac{k}{2\nu_2+1}}x_2, x_3)\to 0
\end{equation}
as $k\to \infty.$
We mention that    $\chi_1(x_1, x_2)$ is supported  in the annulus $D$ and so $|x_1|$ and $|x_2|$ are bounded by 2.  $|x_3| $ is also bounded on the support of $\eta_k(x)$.   Using these facts  we get that  the limit \eqref{phrk} holds uniformly with respect to $(x_1,  x_2, x_3)$ on the support of $\eta_k(x)\chi_1(x_1, x_2)$.
 Let $m(\phi_k)$ denotes the maximal order of vanishing of $\phi_k$ in the support of $\eta_k(x)\chi_1(x_1, x_2)$. We now compute $m(\phi_k)$:  Let  $(x_1^0, x_2^0, x_3^0)$ be any point in the support of $\eta_k(x)\chi_1(x_1, x_2)$ such that $\phi_k(x_1^0, x_2^0, x_3^0)=0.$ If $x_1^0\neq 0,$ then we compute
 $$\partial^2_{x_1x_2}\phi_k(x_1^0,x_2^0,x_3^0)
 =\nu_2(x_1^0)^{\nu_2-1}\left(1+{c_2^{-1}}(\tilde{Q}_2(2^{-\frac{k}{\nu_2+1/2}}x_1^0)-c_2)+2^{-\frac{k}{\nu_2+1/2}}c_2^{-1}\tilde{Q}'_2(2^{-\frac{k}{\nu_2+1/2}}x_1^0)\right).$$
Note that the last two terms in the bracket converges uniformly to zero as $k\to \infty.$  Therefore, if $x_1^0\neq 0$ and $k_0$ is sufficiently large (i.e., $U$ is sufficiently small), then $ \partial^2_{x_1x_2}\phi_k(x_1^0,x_2^0,x_3^0)\neq 0.$ This means the order of vanishing of $\phi_k$ at $(x_1^0, x_2^0, x_3^0)$ is 1.
If $x_1^0=0,$ then $x_2^0\neq 0$ and so  $\partial_{x_1}^{\nu_2}\phi_k(0, x_2^0, x_3^0)=x_2^0\neq 0.$ We also have $\partial_{x_1}^{\alpha_1}\partial_{x_2}^{\alpha_2}\partial_{x_3}^{\alpha_3}\phi_k(0, x_2^0, x_3^0)=0$ for any multi-index $(\alpha_1,\alpha_2,\alpha_2)$ with $\alpha_1+\alpha_1+\alpha_3<\nu_2.$ This means the order of vanishing of $\phi_k$ at $(0, x_2^0,x_3^0)$ is $\nu_2$ and so  $m(\phi_k)=\nu_2.$
If $k_0$ is sufficiently large,  then \cite[Theorem 7.1]{I.U} (see \cite[Proposition 4.5]{I.K.M} for a similar result) implies that, for $p>\nu_2,$
\begin{eqnarray*}
\big|\big|\sup_{t>0}|T^{-k} \mathcal{A}_t^k T^kf|\big|\big|_{L^p}\leq C 2^{k\left(\frac{1}{ p}-\frac{3}{2\nu_2+1}\right)}||f||_{L^p},
\end{eqnarray*}
 where $C>0$ is independent of $k$. Since $T^k$ acts isometrically on $L^p(\mathbb{R}^4),$ this is equivalent to
\begin{eqnarray*}
\big|\big|\sup_{t>0}| \mathcal{A}_t^k f|\big|\big|_{L^p}\leq C 2^{k\left(\frac{1}{ p}-\frac{3k}{2\nu_2+1}\right)}||f||_{L^p}.
\end{eqnarray*}
Since $\frac{1}{ p}-\frac{3}{2\nu_2+1}$ is strictly less than zero, we can sum over all $k\geq k_0$ and obtain the desired estimate for $\mathcal{M}.$

 \end{proof}

\bigskip

 \begin{proof}[\textbf{Proof of Proposition \ref{prop:max}}]
 The first part of the proposition follows from Theorem \ref{th:osc} and \cite[Theorem 1.12]{I.K.M}.

We now prove the second part of the proposition. By Theorem \ref{Bon} there exists an invertible matrix $A\in \mathbb{R}^{3\times 3}$ such that
$ \phi(Ax)$ either depends on  at most two variables $x_1$ and $x_2$ or has the form in \eqref{form}.
Hence, we divide the proof into two parts.

\bigskip
\textbf{Part 1:} Assume that $\phi(Ax) $ depends on at most two variables $x_1$ and $x_2$. Then we have
 \begin{eqnarray*}
 \int_U |\phi(x)|^{-\frac{1}{p}}dx=\det(A) \int_{\tilde{U}} |\phi(Ax)|^{-\frac{1}{p}}dx=\det(A) \int_{\tilde{U}} |\tilde{\phi}(x_1,x_2)|^{-\frac{1}{p}}dx_1dx_2dx_3,
 \end{eqnarray*}
 where $\tilde{U}\colonequals \{Ax:\, x\in U\}$ and $\tilde{\phi}(x_1,x_2)\colonequals \phi(Ax).$ Since $\tilde{\phi}$ does not depend on $x_3$, we can find a constant $C>0$ and a neighborhood $\Omega\subset \mathbb{R}^2$ of the origin such that
 \begin{eqnarray*}
 \int_{\tilde{U}} |\tilde{\phi}(x_1,x_2)|^{-\frac{1}{p}}dx_1dx_2dx_3\geq C\int_{\Omega}  |\tilde{\phi}(x_1,x_2)|^{-\frac{1}{p}}dx_1dx_2.
 \end{eqnarray*}
 Then the proof follows from \cite[Proposition 1.7]{I.K.M}.

 \bigskip
\textbf{Part 2:} Let
$ \phi(Ax)$ have the form in \eqref{form}
\begin{equation*}\label{secondcase}
\phi(Ax)=Q_1(x_1)+Q_2(x_1)x_2+Q_3(x_1)x_3,
\end{equation*}
where $Q_1, Q_2, Q_3:\mathbb{R}\to \mathbb{R}$ are polynomial functions.
As before  let $\nu_1,\nu_2,\nu_3$   denote the multiplicities of $x_1=0$ as the root  of $Q_1(x_1),$ $Q_2(x_1),$ and $Q_3(x_1).$   The conditions $\phi(0)=0$ and $\nabla \phi(0)=0$ imply $\nu_1\geq 2,\nu_2\geq 1,$ and $\nu_3\geq 1.$
If  $Q_2(x_1)$ or  $Q_3(x_1)$ is identically zero, then $\phi(Ax)$ depends at most two variables and so Part 1  implies the proof. Hence, we assume that $Q_2(x_1)$ and $Q_3(x_1)$ are not identically zero. $Q_1(x_1)$ can be identically zero. In this case we formally put $\nu_1=\infty.$
  Without loss of generality we can assume $\nu_2\leq \nu_3.$ Then  $\min\{\nu_1,\nu_2,\nu_3\}$  can be $\nu_1$ or $\nu_2.$
Hence, we consider two different cases:

\textbf{Case 1:} Let $\nu_1= \min\{\nu_1,\nu_2,\nu_3\}.$
 The proof of Theorem \ref{th:adap} shows that in this case the given coordinate system is adapted
and  $d_x=h(\phi)=\nu_1\geq 2.$ Then we can write $\phi(Ax) = x_1^{\nu_1} \tilde{Q}(x_1,x_2,x_3)$, where $\tilde{Q}$ is a polynomial function with $\tilde{Q}(0,0,0) \neq 0$.
We have
 \begin{align*}\int_U |\phi(x)|^{-\frac{1}{p}}dx&=\det(A) \int_{\tilde{U}} |\phi(Ax)|^{-\frac{1}{p}}dx
=\det(A) \int_{\tilde{U}} \frac{dx}{|x_1|^{\frac{\nu_1}{p}}|\tilde Q(x_1, x_2, x_3)|^{\frac{1}{p}}},
 \end{align*}
 where $\tilde{U}\colonequals \{Ax:\, x\in U\}$. If $U$ is sufficiently small, then  $\tilde Q$ is non-zero in $U$. Hence, there is a constant $C>0$ and an interval $(a,b)$ containing $0$ such that
 $$\int_U |\phi(x)|^{-\frac{1}{p}}dx\geq C \int_a^b \frac{1}{|x_1|^{\frac{\nu_1}{p}}}dx_1.$$ Obviously, the integral on the right hand side of this inequality  diverges if $p\leq \nu_1.$

\textbf{Case 2:}  Let $\nu_2= \min\{\nu_1,\nu_2,\nu_3\}.$ If $\nu_1=\nu_2,$ then the proof follows from Case 1. Hence, we assume $\nu_1>\nu_2.$
The proof of Theorem \ref{th:adap} shows that the expression $\phi$ in an adapted coordinate system is
  \begin{align*}
\phi(Bz)=&z_1^{\nu_2}z_2+z_1^{\nu_1}(\tilde Q_1(z_1)-c_1)-\frac{c_1}{c_2}z_1^{\nu_2+1}(\tilde Q_2(z_1)-c_2) \\& +\frac{z_1^{\nu_2}z_2}{c_2}(\tilde Q_2(z_1)-c_2)+\left(z_1^{\nu_3}(\tilde Q_3(z_1)-c_3)-\frac{c_3z_1^{\nu_2} }{c_2}(\tilde Q_2(z_1)-c_2)\right)z_3.
\end{align*}
Then this function can be written as $\phi(Bz)=z_1^{\nu_2}\tilde Q(z_1, z_2, z_3),$ where $\tilde Q$ is a polynomial function satisfying
$\tilde Q(0, 0, 0)=0$ and $\partial_{z_2} \tilde Q(0, 0, 0)\neq0$.
  Thus, we have
 \begin{align*}\int_U |\phi(x)|^{-\frac{1}{p}}dx&=\det(B) \int_{\tilde{U}} |\phi(Ax)|^{-\frac{1}{p}}dx
 =\det(B) \int_{\tilde{U}} \frac{1}{|x_1|^{\frac{\nu_2}{p}}|\tilde Q(x_1, x_2, x_3)|^{\frac{1}{p}}}dx
 \end{align*}
 where $\tilde{U}\colonequals \{Bx:\, x\in U\}$.
 There exist  an interval $(a,b)$ containing $0$ and an neighborhood $\Omega\subset \mathbb{R}^2$ of the origin such that
 $$\int_U |\phi(x)|^{-\frac{1}{p}}dx=\det(B)\int_{a}^b\frac{dx_1}{|x_1|^{\frac{\nu_2}{p}}}\int_{\Omega}  \frac{dx_2dx_3}{|\tilde Q(x_1, x_2, x_3)|^{\frac{1}{p}}}.$$
Obviously, the last integral diverges if $p\leq \nu_2$ and converges whenever $p>\nu_2$.

  \end{proof}

\bigskip

\begin{proof}[\textbf{Proof of Theorem \ref{th:h<2}}]
When the rank of $D^2\phi(0)$ is two, the claimed result follows from \cite[Theorem 2]{Gr}.

We now prove the case $D^2 \phi(0)=0.$
By Theorem \ref{Bon} there exists an invertible matrix $A\in \mathbb{R}^{3\times 3}$ such that
$ \phi(Ax)$ either depends on at most two variables $x_1$ and $x_2$ or has the form in \eqref{form}.
 In the later case, one can check that, if $D^2 \phi(0)=0$, then the height $h(\phi)$ can not be strictly less than 2. Hence, it remains to prove the case where $ \phi(Ax)$  depends on at most two variables $x_1$ and $x_2.$


 Assume $ \phi(Ax)$ depends on at most two variables $x_1$ and $x_2$. By following the proof Theorem \ref{th:Max}, we can present the averaging operator as
$$\mathcal{A}_t f(y)=\int_a^b T_{\theta}R^{\theta}{A}_t^{\theta}R^{-\theta}T_{\theta}^{-1} f(y)d\theta. $$
We remind the operator ${A}_t^{\theta}$
$$ {A}_t^{\theta} f(y)=
\int_{\mathbb{R}^2}f(y_1-tx_1,y_2-tx_2,y_3, y_4-t(1+ \phi(x_1,x_2))) \tilde{\eta} (x_1,x_2, \theta)dx_1dx_2.$$
 By considering $y_3$ and $\theta$ as  parameters, we can apply \cite[Theorem 1.2]{P1} if $U$ is sufficiently small, and obtain that, for $p>h(\phi),$
\begin{equation*}\int_{\mathbb{R}^3}\big(\sup_{t>0}|A^{\theta}_tf(y)| \big)^pdy_1dy_2dy_4\leq C_p \int_{\mathbb{R}^3}|f(y_1,y_2,y_3,y_4)|^pdy_1dy_2dy_4,
\end{equation*}
where $C_p$ is independent of  $\theta$ and $y_3.$
 Integrating this equation with respect to  $y_3,$ we obtain that the operator $\sup_{t>0}|{A}_t^{\theta}f|$ is bounded on $L^p(\mathbb{R}^4)$ for $p> h(\phi).$ Using this fact and repeating the same arguments in the proof of Theorem \ref{th:Max}, we obtain that $\mathcal{M}$ is bounded on $L^p(\mathbb{R}^4)$ for $p>h(\phi).$ The sharpness of the exponent $p>h(\phi)$ follows from Corollary \ref{cor:max}.

\end{proof}

\section{Oscillation and contact indexes}


In analogy with
Arnold's notion of the \emph{singularity index} \cite{Ar, agvMN82}, one can define (see \cite{I.K.M}) the \emph{uniform oscillation
index} $\beta_u(x^0, S)$ of the hypersurface $S\subset \mathbb{R}^{n+1}$ at the point $x^0\in S$ as follows: Let $\mathcal{B}_u(x^0, S)$ denote the set of all $\beta\ge0$ for which there exists an open neighborhood
$U_\beta$ of $x^0$ in $S$ such that estimate
\begin{equation}\label{(1.9)}
|\widehat{\rho d\sigma}(\xi)|\le \frac{C_{\rho, \beta}}{(1+|\xi|)^\beta}
\end{equation}
holds true for every function $\rho\in C_0^\infty (U_\beta)$. Then the uniform oscillation
index $\beta_u(x^0, S)$ is defined by
\begin{eqnarray*}
\beta_u(x^0, S)\colonequals \sup\{\beta: \beta\in  \mathcal{B}_u(x^0, S)\}.
\end{eqnarray*}
If we restrict our attention to the normal direction to $S$ at $x^0$ only, then we can define
analogously the notion of \emph{oscillation index} of the hypersurface $S$ at the point $x^0\in S$.
More precisely, if $n(x^0)$ is a unit normal to $S$ at $x^0$, then we let $\mathcal{B}(x^0, S)$ denote the set of all $\beta\ge0$ for which there exists an open neighborhood
$U_\beta$ of $x^0$ in $S$ such that estimate \eqref{(1.9)} holds true along the line $\mathbb{R}n(x^0)$ for every function $\rho\in C_0^\infty(U_\rho)$, i.e.,
\begin{eqnarray*}
|\widehat{\rho d\sigma}(\lambda n(x^0))|\le \frac{C_{\rho, \beta}}{(1+|\lambda|)^\beta}.
\end{eqnarray*}
Then the oscillation index $\beta(x^0, S)$ is defined by
\begin{eqnarray*}
\beta(x^0, S)\colonequals \sup\{\beta: \beta\in  \mathcal{B}(x^0, S)\}.
\end{eqnarray*}

We also define the \emph{uniform contact index}
$\gamma_u(x^0, S)$ of the hypersurface $S$ at the
point $x^0\in S$ as follows: Let $\mathcal{C}_u(x^0, S)$ denote the set of all
$\gamma>$0 for which there exists an
open neighborhood $U_\gamma$
 of $x^0$ in $S$ such that the estimate
\begin{equation*}
\int_{U_\gamma}d_H(x)^{-\gamma} d\sigma(x)<\infty
\end{equation*}
holds true for every affine hyperplane $H$ in $\mathbb{R}^n$. Then we put
\begin{eqnarray*}
\gamma_u(x^0, S):=\sup\{\gamma: \gamma\in  \mathcal{C}_u(x^0, S) \}.
\end{eqnarray*}

Similarly, we let $\mathcal{C}(x^0, S)$ denote the set of all
$\gamma>0$ for which there exists an open neighborhood
$U_\gamma$ of $x^0$ in $S$ such that
\begin{equation*}
\int_{U_\gamma}d_{T_{x^0}}(x)^{-\gamma} d\sigma(x)<\infty,
\end{equation*}
where $T_{x^0}$ affine tangent hyperplane at the point $x^0$
and call
\begin{eqnarray*}
\gamma(x^0, S) :=\sup\{
\gamma: \gamma\in  \mathcal{C}(x^0, S)\}.
\end{eqnarray*}
the contact index
$\gamma(x^0, S)$ of the hypersurface $S$ at the point $x^0\in S$. Then clearly
\begin{eqnarray*}
\beta_u(x^0, S)\le \beta(x^0, S)\quad \mbox{ and}\quad
\gamma_u(x^0, S)\le \gamma(x^0, S).
\end{eqnarray*}

The main finding of this paper related to estimates for the Fourier transform of surface-carried measures, can be summarized as follows:

\begin{theorem}\label{Main}
Let $S$ be a  hypersurface in $\mathbb{R}^4$ and  $x^0\in S.$ After applying a suitable Euclidean motion of $\mathbb{R}^4,$ let us assume that $x^0:=(0, 0, 0, 0)\in S$ and in a neighborhood of  $x^0$ we may view $S$ as the graph of a polynomial function $\phi$ satisfying  $\phi(0)=0, \nabla\phi(0)=0$. Assume the   Gaussian curvature of $S$ is identically zero in this neighborhood of $x^0\in S$. Let $h(x^0, S)\colonequals h(\phi)$ be the height of $\phi$ defined in \eqref{height}. 
Then the following relations
\begin{equation}\label{maineq}
\beta_u(x^0, S)=\beta(x^0, S)=\gamma_u(x^0, S)=\gamma(x^0, S)=\frac{1}{h(x^0, S)}.
\end{equation}
hold.
\end{theorem}
\begin{proof}
The proof  follows by using  Theorem \ref{th:osc}, Theorem \ref{th:Max}, and  Proposition \ref{prop:max} and by repeating the same arguments in the proof of \cite[Theorem 1.14]{I.K.M}.
\end{proof}
\begin{Remark}
 Theorem \ref{Main} is analogous to Theorem 1.14 in \cite{I.K.M}. It should be noted that the examples of Varchenko in  \cite{Var}  show that, in general,  the height $h(x^0, S)$ does not necessarily determine $\beta(x^0, S)$. Also, in some cases the strict inequality $\beta_u(x^0, S) < \beta(x^0, S)$ holds, for more simple examples, see  \cite{Gressman}. However, these examples do not refute the conjecture proposed by Stein, Iosevich, and Sawyer.

\end{Remark}

\section{Extensions}
Let $\phi:\mathbb{R}^3\to \mathbb{R}$  be a smooth function of finite type at the origin (see Section 2) with $\phi(0)=0$ and $\nabla \phi(0)=0$.  Assume that  there is an invertible matrix $A\in \mathbb{R}^{3\times 3}$ such that $\phi(Ax)$ either depends on at most two variables $x_1$ and $x_2$ or has the form
$$\phi(Ax)=\phi_1(x_1)+\phi_2(x_1)x_2+\phi_3(x_1)x_3,$$
where $\phi_i:\mathbb{R}\to \mathbb{R},\, i\in \{1,2,3\}$ are smooth functions. For such functions, one can check that
$$\det(D^2\phi(x))=0,\,\,\,\,\forall\,x\in \mathbb{R}^3.$$
Theorem \ref{Bon} shows that all polynomials with vanishing Hessian determinant have  this property.

All of our results can be extended for all smooth  functions  $\phi$ satisfying the above condition.
  We comment on how to construct an adapted coordinate system: Let $\phi(Ax)$ depends on at most two variables, then, as in Part 1 of the proof of Theorem \ref{th:adap},  the existence of an adapted coordinate system follows from \cite{Adap}.
Let $\phi(Ax)=\phi_1(x_1)+\phi_2(x_1)x_2+\phi_3(x_1)x_3,$ then the conditions $\phi(0)=0$ and $\nabla \phi(0)=0$ imply $\phi_1(0)=\partial_{x_1}\phi_1(0)=0$ and $\phi_2(0)=\phi_3(0)=0.$
 Since we assume $\phi$ is finite type, at least one of $\phi_1,$ $\phi_2$, and $\phi_3$ is not flat at the origin. Hence, there exists a finite natural number $m\geq 1$ such that $ \partial_{x_1}^{k}\phi_1(0)=0$ for all $k\in \{0,...,m\}$, $\partial_{x_1}^{j}\phi_2(0)= \partial_{x_1}^{j}\phi_3(0)=0$ for all $j\in \{0,...,m-1\}$,   and at least one of $\partial^{m+1}_{x_1}\phi_1(0),\,\partial^{m}_{x_1}\phi_2(0),$ and $\partial^{m}_{x_1}\phi_3(0)$ is non-zero.
Then we can write \begin{align*}
\phi(Ax)= x_1^{m}\left(\frac{\partial^{m+1}_{x_1}\phi_1(0)}{(m+1)!}x_1+\frac{\partial^{m}_{x_1}\phi_2(0)}{m !}x_2+\frac{\partial^{m}_{x_1}\phi_3(0)}{m!}x_3\right)
+\tilde{\phi}_1(x_1)+\tilde{\phi}_2(x_1)x_2+\tilde{\phi}_3(x_1)x_3,
\end{align*}
where $\tilde{\phi}_1(x_1)\colonequals {\phi}_1(x_1)-\frac{\partial^{m+1}_{x_1}\phi_1(0)}{(m+1)!}x_1^{m+1}$, $\tilde{\phi}_2(x_1)\colonequals {\phi}_2(x_1)-\frac{\partial^{m}_{x_1}\phi_2(0)}{m !}x_1^{m},$ and $\tilde{\phi}_3(x_1)\colonequals {\phi}_3(x_1)-\frac{\partial^{m}_{x_1}\phi_3(0)}{m!}x_1^{m}.$
If $\partial^{m+1}_{x_1} \phi_1(0) \neq 0$ and $\partial^{m}_{x_1}\phi_2(0)=\partial^{m}_{x_1} \phi_3(0)=0,$ then
$$\phi(Ax)= \frac{\partial^{m+1}_{x_1}\phi_1(0)}{(m+1)!}x_1^{m+1}
+\tilde{\phi}_1(x_1)+\tilde{\phi}_2(x_1)x_2+\tilde{\phi}_3(x_1)x_3$$
 and  the given coordinate system $x$ is adapted. The distance is  $d_x=m+1$.   If at least one of $\partial^{m}_{x_1}\phi_2(0)$ and $\partial^{m}_{x_1}\phi_3(0)$ is non-zero, then an adapted coordinate system can be obtained by
\begin{align*}
z_1 &=x_1,\\
z_2 &= \frac{\partial^{m+1}_{x_1}\phi_1(0)}{(m+1)!}x_1+\frac{\partial^{m}_{x_1}\phi_2(0)}{m !}x_2+\frac{\partial^{m}_{x_1}\phi_3(0)}{m!}x_3,\\
z_3&=x_3.
\end{align*}
In this new coordinate system $z$, $\phi$ has the form
$$z_1^{m}z_2+\psi_1(z_1)+\psi_2(z_1)z_2+\psi_3(z_1)z_3$$
for some smooth function $\psi_1(z_1),$ $\psi_2(z_1),$ and $\psi_3(z_1),$ which satisfy    $ \partial_{x_1}^{k}\psi_1(0)=0$ for all $k\in \{0,...,m+1\}$ and $\partial_{x_1}^{j}\psi_2(0)= \partial_{x_1}^{j}\psi_3(0)=0$ for all $j\in \{0,...,m\}$ .
 The distance in the coordinate system $z$ is $d_z=m.$ The adaptness of these coordinate systems can be proven by considering any smooth local coordinate system $y$ as in \eqref{new.coor} and using the arguments in Part 2 of the proof of Theorem \ref{th:adap}.   

 Also, our results concerning oscillatory integrals in Theorem \ref{th:osc}  and maximal operators in Theorem \ref{th:Max} can be obtained for such functions by using the  arguments in the proofs of these theorems.

 \bigskip

Getting rid of the above
assumption on $\phi$ would of course be interesting, but since our results are, to the best of our knowledge, the first sharp results in three dimensions, we believe that they have their own interest.

 As an extension of the present work, one could study smooth functions whose Hessian determinant is identically zero. Yet, this extension is not within reach so far and will be a matter of further study.

 In three dimensions and higher, sharp estimates for asymptotic behavior of oscillatory integrals and  maximal operators on $L^p$ remain unknown for general hypersurfaces. This paper contributes to the  three-dimensional case. We expect that the techniques used in this paper can
be useful for future studies.

\bigskip

\textbf{Acknowledgement.}
 G.\,Toshpulatov is funded  by the Deutsche Forschungsgemeinschaft (DFG, German Research Foundation) under Germany's Excellence Strategy EXC 2044/2--390685587, Mathematics M\"unster: Dynamics--Geometry--Structure.\\

\textbf{Data Availability.} There is no data associated with the paper.\\

\textbf{Conflict of interest.} The author has no conflict of interest to declare.

{}
\end{document}